\documentclass[a4paper,11pt]{article}
\usepackage[utf8]{inputenc}

\usepackage{array}

 \usepackage{amsthm,amsmath,amsxtra,amscd,amssymb,xypic,color,mathtools,xspace}
\usepackage{todonotes}
\usepackage{enumerate}
\usepackage{multirow}
\usepackage{tikz}

 \usepackage{gfsartemisia}

 \usepackage{hyperref}
\usepackage{ucs}
\usepackage[english]{babel}
\usepackage[T1]{fontenc}

\usepackage{caption}
\usepackage{subcaption}

\usepackage{amssymb,amsfonts,amsmath,amsthm,mathrsfs}
\usepackage{paralist}
\usepackage[all]{xy}
\usepackage{graphicx}
\usepackage{color}
\usetikzlibrary{calc,matrix,arrows,shapes,decorations.pathmorphing,decorations.markings,decorations.pathreplacing}
\usepackage{xifthen}
\usepackage{verbatim}
\usepackage{xspace}

\usepackage{csquotes}
\usepackage[]{appendix}

 \newcommand{\NN}{{\mathbb N}} 
\newcommand{\ZZ}{{\mathbb Z}} 

\newcommand{\RR}{{\mathbb R}} 
\newcommand{\CC}{{\mathbb C}}

 \newcommand{\calO}{{\mathcal O}}
 \newcommand{\hor}{\operatorname{hor}}
 \newcommand{\ord}{\operatorname{ord}}
 
  \newcommand{\SL}{\operatorname{SL}}
  \newcommand{\Id}{\operatorname{Id}}
    
     \newcommand{\Br}{\operatorname{Br}}

 \newcommand{\pgcd}{{\rm pgcd}}
 \newcommand\Res{\operatorname{Res}}
 \newcommand\Ext{\operatorname{Ext}}
 \newcommand\Reel{\operatorname{Re}}
 \newcommand\Imagin{\operatorname{Im}}
 \newcommand\Aff{\operatorname{Aff}}

 \newcommand{\moduli}[1][g]{{\mathcal M}_{#1}}

 \newcommand{\komoduli}[1][g]{{\Omega^{k}\mathcal M}_{#1}}
 
\theoremstyle{plain}{
 \newtheorem{thrm}{Theorem}[section]
\newtheorem{cor}[thrm]{Corollary}
\newtheorem{lmm}[thrm]{Lemma}
\newtheorem{prop}[thrm]{Proposition}

}

\theoremstyle{remark}{
\newtheorem{rmk}[thrm]{Remark}
}

\theoremstyle{definition}{
\newtheorem{defn}[thrm]{Definition}
\newtheorem{conv}[thrm]{Convention}
\newtheorem{ex}[thrm]{Example}

}

\newcommand{\Weierstrass}{Weierstra\ss\xspace}
 
 \hypersetup{pdfborder=0 0 0, 
	    colorlinks=true,
	    citecolor=purple,
	    linkcolor=violet,
	    urlcolor=black,
	    pdfauthor={Andrei Bogatyrev and Quentin Gendron}
	   }

\def\be{\begin{equation}}
\def\ee {\end{equation}}
\def\ord{{\rm ord}}

\def\vert{{\rule{.5mm}{2.mm}}}
\def\hor{{\rule{2.mm}{.5mm}}}

\graphicspath{{figs/}}  
  
\title{The space of solvable Pell-Abel equations}

 \author{Andrei Bogatyr\"ev and Quentin Gendron}

\begin{document}

 \maketitle

 \selectlanguage{english}

\begin{abstract}
Pell-Abel equation is a functional equation of the form $P^{2}-DQ^{2} = 1$, with a given polynomial $D$ free of squares and unknown polynomials $P$ and $Q$. We show that the space of Pell-Abel equations with the fixed degrees of $D$  and of a primitive solution~$P$ is a complex manifold. We describe its connected components by an efficiently computable invariant. Moreover, we give various applications of this result, including torsion pairs on hyperelliptic curves,  Hurwitz spaces and the description of the connected components of the space of primitive $k$-differentials with a unique zero on genus $2$ Riemann surfaces.
\end{abstract}


 \section{Introduction}

The reincarnation of the Diophantine equation of Pell  in the realm of
polynomials
was introduced and investigated by N. H. Abel in \cite{Abel}. Since
then the equation
\be
\tag{PA}
P^2(x)-D(x)Q^2(x)=1
\label{PA}
\ee
is known as \emph{Pell-Abel equation}.  Here $P(x)$ and $Q(x)$ are
unknown polynomials of one variable and
$D(x):=\prod_{e\in{\sf E}}(x-e)$ is a given degree $\deg D=|{\sf
E}|:=2g+2$ monic complex polynomial  without multiple roots. For a
generic choice of
$D$, the Pell-Abel equation only admits the trivial solutions
$(P,Q)=(\pm1,0)$.  If an equation has a nontrivial solution then  the
set of solutions is
infinite and contains a unique, up to sign, polynomial  with minimal
degree $n:=\deg P>0$. This solution is called \emph{primitive}. It
generates the other  solutions~$P$ via composition with the classical
Chebyshev polynomials and change of sign. This is discussed in
more details in Section~\ref{sec:solv}.
\smallskip
\par
Let us fix $g\geq0$ and $n\ge1$ and consider the set
$\tilde{\mathscr{A}}_{g}^{n}$ of monic polynomials $D$ of degree $\deg
D=2g+2$ whose associated Pell-Abel equations  have a primitive
solution of degree~$n$.  The affine group $x \mapsto ax +b$ with $a\in
\CC^{\ast}$ and $b\in \CC$ acts on the set of monic polynomials as $D(x)
\mapsto a^{-\deg D}D(ax+b)$. This action does not affect the degree
$n$ of the primitive solution of Equation~\eqref{PA} and our main
object of study is the quotient~$\mathscr{A}_{g}^{n}$  of the set
$\tilde{\mathscr{A}}_{g}^{n}$ by this group action. More precisely, we
have the following result proved in Section~\ref{sec:ModuliSpace}.
  \begin{thrm}\label{thm:geomstruct}
 The  set $\tilde{\mathscr{A}}_{g}^{n}$  is
 invariant under the action of the affine group. The
quotient~$\mathscr{A}_{g}^{n}$ is a smooth orbifold of complex
dimension $g$.
 \end{thrm}

 A first reading about orbifolds is Section~13 of \cite{ThGT3} and  in
the first approximation we can think of them as manifolds.

The main result of this paper consists in  classification of the
connected components of the spaces $\mathscr{A}_{g}^{n}$.  A weaker version was
announced in \cite{BoGeShort}, which contains a survey of the proof
given below.

 We first introduce the \emph{degree partition invariant} of a on element 
 $D\in\mathscr{A}_g^n$. Given a primitive solution $P$ of Equation~\eqref{PA}, its value $P(e) = \pm1$  at any zero $e\in \sf E$ of~$D$. 
 Therefore the set $\sf E$ is decomposed into two subsets ${\sf E}^\pm$ and we obtain the partition of the degree of $D$:
 $$
 |{\sf E}|=2g+2=|{\sf E}^+|+|{\sf E}^-|.
 $$
 The choice of the other primitive solution $-P$ interchanges indexes 
 $\pm$ in the summands, but the unordered partition remains the same.
 The \emph{degree partition invariant} of $D$ is the unordered pair of nonnegative integers $(|{\sf E}^-|,|{\sf E}^+|)$.

 \begin{thrm}\label{main}
Let $m=\min(g,n-g-1)$ and $[\cdot]$ denotes the integer part.
Equation~\eqref{PA} has no primitive solutions of degree $n<g+1$ or $n>1$ when $g=0$. Otherwise, the number of components $a(g,n)$ of $\mathscr{A}_{g}^{n}$ is equal to $[m/2]+1$ if $n+g$ is odd and $[(m+1)/2]$ if $n+g$ is even.
Moreover, each component is labelled by a unique
 degree partition $(|{\sf E}^-|,|{\sf E}^+|)$ satisfying
\begin{enumerate}
 \item $|{\sf E}^\pm|>0$,
 \item $|{\sf E}^\pm|\leq n$,
  \item the parity of $|{\sf E}^\pm|$ is equal to the parity of $n$.
\end{enumerate}
\end{thrm}

This theorem has two trivial cases. When $n<g+1$, the degree of $P^2$
is strictly less than the degree of $DQ^2$ if the solution $(P,Q)$ is not trivial. When $g=0$, any Equation~\eqref{PA}
is brought to the case  $D(x)=x^2-1$  by a linear change of variable.
It admits the (primitive) solution $(P,Q)=(x,1)$ of degree $n=1$.
All the other cases are far less trivial. They are based on a
pictorial calculus representing the flat structure on the Riemann surface that we associate to each Pell-Abel equation.
\smallskip
\par
First, in Section~\ref{sec:solv}, we associate to every (marked) hyperelliptic Riemann surface a distinguished abelian differential. Using it, we propose a solvability criterion for the 
Pell-Abel equation in terms of the periods of this differential.
Then in Section~\ref{sec:graphcalcul} we elaborate the graphical technique
which allows to control the periods of the distinguished differential when we deform the polynomial~$D$. The upper bound for the number of connected components is obtained  in  Section~\ref{sec:upperbound}, where we bring
the graph of an arbitrary solvable Pell-Abel equation to one of standard forms.  Finally, we discuss the degree partition invariant  in Section~\ref{sec:GlobInv}. We show that it appears in the context of braids and that all standard forms have different invariants, hence lie in different components.
\smallskip
\par
\paragraph{\bf Applications.}
The Pell-Abel equation is inherently related to many problems in
various branches of
mathematics.
To cite some, it appears
in the reduction of abelian integrals
\cite{Abel, NGCheb, BelEno}, Poncelet porism \cite{ZheBu}, elliptic
billiards \cite{DragoRad},
approximation theory \cite{SYu,Peh,Bbook,B02},
spectral theory for infinite Jacobi matrices \cite {SYu}, algebraic
geometry including the study of
Frobenius endomorphisms \cite{Serre}, complex affine surfaces \cite{KollPell}, Teichmüller curves \cite{mcmtor}, etc.

Now we give some examples where our main result may be directly translated or applied.
\smallskip
\par
{\bf  (A) Extremal polynomials.} 
Shabat polynomials, i.e. polynomials with just two finite critical values, are rigid objects
whereas many applications require maps with similar properties, but more flexible. Those were defined in  
\cite{B02, Bbook} under the name of  {\em $g$-extremal polynomials} and in \cite[Section 12.2.2]{ZannierUnli} or \cite[Section~2]{BSZ22} as {\em almost Belyi maps}. A typical $g$-extremal polynomial $P(x)$ has only simple critical points with almost all critical values equal to $\pm1$ and exactly $g$ exceptional critical values not lying in this set.  In general we allow merging the critical points, and the extremality weight~$g$ defined e.g. in \cite{B02, Bbook} takes into account the confluent critical points,
even if the appropriate critical value lies in the exceptional set $\{\pm1\}$.

A practical interest in $g$-extremal polynomials comes from some
problems of uniform Chebyshev optimisation: the vast majority of arising alternation points for the
solution will be the critical ones and with the  values in the two-element set: $\pm$ the value
of approximation error. After re-normalization they become $g$-extremal with some small value of parameter $g$. Classical examples are Chebyshev and Zolotarev polynomials for $g=0$ and $g=1$ respectively.

Any polynomial $P$ is a solution of the unique Pell-Abel equation:
just extract the square-free part $D$ in the polynomial $P^2-1$. A simple calculation shows that
$\deg D=2g+2$ where $g$ is the extremality number of the polynomial $P$.
The set of $g$-extremal complex polynomials of given degree $N\ge g+1$ is a smooth complex manifold of dimension $g+2$ and the number of its components may be counted with the use of our main theorem. Indeed, every $g$-extremal polynomial $P_N(x)$ as a solution of Pell-Abel equation has the unique representation of the kind $\pm T_m\circ P_n(x)$, where $T_m$ is the classical degree~$m$ Chebyshev polynomial and $P_n$ is the primitive solution of the same Pell-Abel equation (see Theorem~\ref{criterion}).
One can show that the inverse polynomials $\pm P_N$ lie in the same component of the set of  $g$-extremal polynomials exactly when the corresponding degree partition has equal parts: $|{\sf E}^\pm|=g+1$.  Eventually, we arrive at 

\begin{cor}
The deformation space of $g$-extremal polynomials of the given degree $N$ 
consists  of one or two components when $g=0$ and $N$ is respectively odd or even.
For $g>0$ the same number is equal to 
\begin{equation}
 \sum_{n|N} 2a(g,n)-  \#\left\{ n\in\NN : \frac{N}{n} \text{ is odd and } n-g=1,3,5,\dots\right\}\,,
\end{equation}
where $a(g,n)$ is the number of components of  $\mathscr{A}_{g}^{n}$.
\end{cor}

\smallskip
\par
{\bf  (B) Hurwitz spaces.}  
A typical $g$-extremal  polynomial of degree $N$ with different exceptional critical values gives us a covering of a sphere by another sphere which is branched in a specific way. The cyclic type of monodromy
above $g+2$ finite critical points is described by the passport (see \cite{LZ} for definitions)
$$
[2^A1^{N-2A};~2^B1^{N-2B};~g\times 2^11^{N-2}]
$$
with  integers $2\le A,B \le N/2$ satisfying the planarity (or Riemann-Hurwitz) condition $A+B+g=N-1$.
Again, the polynomials $P_N$ realizing the above passport after their re-normalization have the representation as the composition of classical Chebyshev polynomial $T_m$  and a primitive solution $P_n$ of some Pell-Abel equation.  We should distinguish between two cases: for even $m$ the maximum of $2A,2B$ equals to $N$ and the minimum
is equal to $N-2g-2$; for odd $m$ the positive numbers $N-2A$ and $N-2B$ make up the degree partition for $P_n$.
\begin{cor}
The Hurwitz space of degree $N$ polynomials with the above monodromy passport 
has the following number of components:
\begin{enumerate}
 \item The number of integer $n$ such that $N/n$ is odd
and $n\ge N-2\min(A,B)$, when $N>2\max(A,B)$.
\item The sum $\sum_n a(g,n)$ over integer $n$ such that $N/n$ is even, when $N=2\max(A,B)$.
\end{enumerate}
\end{cor}

Note that this generalises works with similar passports as in \cite{WajHur,FuOsHurwitz,MoPiHur} and partial results on these passports in \cite{KhZd} and in \cite[Table 5.1]{LZ}. Moreover, the use of abelian differentials to study Hurwitz spaces already appeared in \cite{mulSecond}.

\smallskip
\par
{\bf  (C) Torsion points.}
Given a genus $g$ hyperelliptic Riemann surface $M$ 
 with hyperelliptic involution
$J$ and a non-\Weierstrass marked point $p$. We can  ask when the Abel-Jacobi image  of the divisor
$p-Jp$ has some finite order $n$ in the Jacobian. Equivalently we can ask about the 
existence of a function $f\in \CC(M)$ whose divisor is $n(p-Jp)$.  This problem is equivalent to
solvability of some Pell-Abel equation which we explain in Remark~\ref{rem:solve} of Section~\ref{sec:solv}. Therefore, we claim that:

\begin{cor}\label{cor:torsion}
 The number of connected components of  the space of hyperelliptic
Riemann surfaces $M$ of genus $g$
 with a primitive $n$-torsion pair of  points conjugated by the
hyperelliptic involution is equal
 to $a(g,n)$. 
The degree partition $(|{\sf E}^-|,|{\sf E}^+|)$ is the number of $e\in \sf E$ such that 
$f(e,0) = \pm 1$ for suitable normalization of this function in the algebraic model 
\eqref{M} of~$M=M({\sf E})$.
\end{cor}

\smallskip
\par
{\bf  (D) Strata of $k$-differentials.}
A more elaborated application is the following result proved in Section~\ref{sec:compkdiffs}, where basic definitions are recalled.
\begin{cor}\label{cor:CCkdiff}
The moduli space of primitive $k$-differentials with a unique zero of order~$2k$ on genus~$2$  Riemann surfaces $\komoduli[2](2k)^{\rm prim}$ is empty for $k=2$, connected for $k=1,3$ or  $k\geq4$ even
and has two connected components for $k\geq5$ odd.  Moreover, the component of~$\mathscr{A}_{g}^{n}$ of degree partition invariant $(1,5)$, resp. $(3,3)$, corresponds to the component of odd, resp. even, parity of $\komoduli[2](2k)^{\rm prim}$.
 \end{cor}
 The proof of the second part of the corollary is given in Proposition~\ref{prop:relationplatcheby} by studying the torsion packets modulo the \Weierstrass points, which may be of independent interest.

\smallskip
\par
\paragraph{\bf Acknowledgements:}
 We thank Vincent Delecroix for the programming help and Victor
Buchstaber for his constant interest in
 this topic. Various aspects of this work were discussed at the
research seminars: \emph{A. Gonchar seminar on complex analysis},
 \emph{Graphs on surfaces and curves over arithmetic fields},
 \emph{HSE math seminar}.
 The authors thank the organizers and all the participants of these
seminars for fruitful discussions.
 Also we thank the anonymous referees of this paper for their valuable
suggestions.
Finally, our special thanks go to Jean-Pierre Serre who initiated our
collaboration.

\section{Solvability of Pell-Abel equation}
\label{sec:solv}

Fix a polynomial $D$ of degree $2g+2$  whose roots are all simple. The union of these roots is denoted by  $\sf E$. Some
conditions on $D$ have to be imposed  \cite{Abel,NGCheb,Mal,SYu} to guarantee the existence of a nontrivial 
solution of Pell-Abel Equation \eqref{PA}, that is with $n:=\deg P>0$. The criterion given by Abel is the periodicity of the continued fraction 
for the square root of $D$ (see \cite{PlatoTorSurv} and the references therein for a more modern presentation). We will use a transcendental criterion coming from
\cite{B02, Bbook} which is much easier to handle with.

We associate to the polynomial $D(x)=\prod_{e\in{\sf E}}(x-e)$ the affine genus $g$ hyperelliptic Riemann surface
\be 
M=M({\sf E}):=\left\{(x,w)\in \mathbb{C}^2: w^2=D(x)\right\}\,.
\label{M}
\ee
The latter admits the natural two point compactification
\be
M_{\infty}:=M\cup \{\infty_\pm\}
\label{Mc}
\ee
where the two points $\infty_\pm$ at infinity are distinguished by the 
limit value of the function $w^{-1}x^{g+1}(\infty_\pm)=\pm1$. 
The added points are interchanged by the hyperelliptic involution $J(x,w)=(x,-w)$
acting on $M_{\infty}$. In what follows, we will suppose that the points $\infty_{\pm}$ are marked on the Riemann surface $M_{\infty}$.

The  Riemann surface  $M_{\infty}$ associated to $D$ bears a unique meromorphic differential of the third kind
\be
d\eta=d\eta_M=(x^g+a_{g-1}x^{g-1}+\dots +a_{0})w^{-1}dx
\label{deta}
\ee 
having two simple poles at infinity with residues $\Res d\eta|_{\infty_\pm}:= \mp1$ and purely imaginary periods (see Proposition~3.4 of \cite{GrKr}). This differential will be referred as the \emph{distinguished differential}. 

Note that the distinguished differential is odd with respect to hyperelliptic involution: $J^*d\eta=-d\eta$. In particular, there is a unique quadratic differential $(d\eta)^{2}$ on the Riemann sphere such that $d\eta$ is the root of the pull-back of $(d\eta)^{2}$ on~$M_{\infty}$ (or $d\eta$ is the canonical cover of $(d\eta)^{2}$ in the terminology of~\cite{BCGGM3}). This quadratic differential is referred as the {\em distinguished quadratic differential}.

We give the criterion of solvability of Equation~\eqref{PA} in terms of the distinguished differential.

\begin{thrm}
\label{criterion}
Given $n\geq1$, Equation~\eqref{PA}  admits a nontrivial solution with $\deg P=n$ if and only if all the periods of  $d\eta_M$ on $M$ are contained in the lattice $2\pi i\mathbb{Z}/n$. 

If this condition is satisfied, then
the solution of Pell-Abel equation is given, up to sign, by:
\begin{equation}\label{PQ}
 P (x)=  \cos \left( ni\int_{(e,0)}^{(x,w)}d\eta_M \right) \text{ and } Q(x) = iw^{-1} \sin\left( ni\int_{(e,0)}^{(x,w)}d\eta_M \right) \,.
\end{equation}
\end{thrm}

\begin{proof}
 If Equation~\eqref{PA} has a nontrivial solution $(P,Q)$ then 
the (Akhiezer) rational function $f(x,w)=P(x)+wQ(x)\in \CC(M_{\infty})$ satisfies $f(x,-w) =1/f(x,w)$.
Hence it has a unique pole at $\infty_+$ and a unique zero at $\infty_-$, both of multiplicity $n$. In that case, the distinguished differential is equal
to $d\eta=n^{-1}d\log(f(x,w))$. The fact that $\log$ is $2i\pi$-periodic implies that the periods of $d\eta$ lie in $2i\pi \mathbb{Z}/n$. 

Conversely, the lattice condition 
\be
\int_{H_1(M,\ZZ)} d\eta_M\subset2\pi i\ZZ/n
\label{LatCond}
\ee
and the fact that  $J^{\ast}d\eta = -d\eta$ imply that the functions in the right hand sides of Equation~\eqref{PQ} are polynomials of degrees $n$ and $n-g-1$ respectively. 
Now the Pythagorean theorem $\sin^2(z)+\cos^2(z)=1$ for $z\in \CC$  reads as Pell-Abel equation.
\end{proof}

\begin{rmk}\label{rem:solve}
1) The lattice condition as the criterion for the solvability of Pell-Abel equation first appeared seemingly 
in approximation theory and it is related to Chebyshev approach to least deviation problems \cite{Zol, Bbook}. 
Some particular cases may be found in \cite{Rob, SYu, Peh, B99, B02}. 

2) Given a polynomial $D$, the set of all solutions of Equation~\eqref{PA} admits a group structure which mimics 
multiplication of Akhiezer functions:
\be
(P,Q)*(p,q)=(Pp+DQq,Pq+Qp).
\ee
The trivial solution $(1,0)$ is the unit of this group and the inverse of $(P,Q)$ is $(P,-Q)$. It follows from the trigonometric representation of solutions given in Equation~\eqref{PQ}, that the primitive solution generates all higher degree solutions
via the composition with the classical Chebyshev polynomial and possibly a change of sign.

3) Note that if Equation~\eqref{PA} has a non trivial solution $(P,Q)$ of degree $n$, then the Akhiezer function  $f(x,w)=P(x)+wQ(x)\in \CC(M_{\infty})$ has divisor $n\infty_{+} -n\infty_-$. Hence the divisor $\infty_{+}-\infty_{-}$ is of primitive $n$-torsion if and only if Equation~\eqref{PA} has a primitive solution of degree~$n$. Corollary~\ref{cor:torsion} follows readily from Theorem~\ref{main} using this remark.
\end{rmk}

\section{Space of Pell-Abel equations}
\label{sec:ModuliSpace}

Let us study the constraints imposed by the lattice condition~\eqref{LatCond} of Theorem~\ref{criterion}. 
Consider the space $\tilde{\cal H}_g$ of complex monic square free polynomials $D(x)$ of degree $2g+2$. It may be identified
with the space $\CC^{2g+2}$ with removed discriminant set. The disjoint zeros $e\in\sf E$ may serve as local 
coordinates of this complex manifold. The polynomials such that the Pell-Abel Equation~\eqref{PA} 
has a \emph{primitive} solution of degree $n\geq1$ form a subset~$\tilde{\mathscr{A}}_{g}^{n}$ of~$\tilde{\cal H}_g$.  We show it is a manifold.

\begin{thrm}
The set of polynomials $\tilde{\mathscr{A}}_{g}^{n}$ is either empty or a smooth complex manifold of pure dimension $g+2$.
\label{PAmanifold}
\end{thrm}

The proof relies on the fact that the set $\tilde{\mathscr{A}}_{g}^{n}$ is given by the polynomials $D(x)$ such that the associated distinguished differential $d\eta$ on 
$M$ satisfies the lattice condition~\eqref{LatCond} of Theorem~\ref{criterion}. 

\begin{proof}
Consider the space of non-normalised abelian differentials
$$
d\eta({\sf B,E}):=\frac{\left(x^g+\sum_{s=0}^{g-1}b_sx^s\right)}{\sqrt{\prod_{j=1}^{2g+2}(x-e_j)}}dx\,,
$$
with coordinates $({\sf B,E}):=(b_0, \dots,b_{g-1}; e_1,\dots,e_{2g+2})$. This is a natural fibration over 
the space $\tilde{\cal H}_g$.

Let us fix $2g+1$ closed paths on the given twice punctured surface $M=M(\sf E_0)$ which represent a basis of the homology group $H_1(M,\ZZ)$ (an extra nontrivial cycle encompasses a puncture). By disturbing the loops within their homology class we suppose that the projections $C_0,C_1, 
\dots, C_{2g}$ of those contours to the $x$-plane  are disjoint from the branching set ${\sf E}_0$. 
Therefore for all ${\sf E}\in \tilde{\cal H}_g$ in a small vicinity of  ${\sf E}_0$ the lifts of those contours 
to the surface $M({\sf E})$ represent the basis of the first homology group. We denote by $C_0$ the cycle 
encompassing a puncture at infinity.

We also fix $g+2$  paths $D_s$ on the complex plane disjoint from the branching set 
${\sf E}_0$, starting at a common point  $p_0$ and ending at arbitrarily chosen but distinct points $p_s$,
for $s=1,\dots, g+2$. Finally, we fix a loop $D_0$   lifting to an open path on $M({\sf E}_0)$ and connecting two preimages of $p_0$ on the surface. This set of data provides us with $3g+2$ locally defined   holomorphic functions:
\begin{eqnarray*}
 \pi_j({\sf B,E}) &:=&\int_{C_j}d\eta({\sf B,E}), \text{ for } j=1,2,\dots,2g\,, \text{ and }\\
\tau_s({\sf B,E})&:=&\int_{D_s}d\eta({\sf B,E})+\frac12\int_{D_0} d\eta({\sf B,E}),  \text{ for } s=1,2,\dots,g+2 \,.
\end{eqnarray*}

If the coordinate change $(\sf B,E)\to(\pi,\tau)$ is degenerate at the point $(\sf B_0,E_0)$, there exists a 
tangent vector $\sum_j\beta_j\frac\partial{\partial b_j}+\sum_s\epsilon_s\frac\partial{\partial e_s}$ 
annihilating all these functions at this point of the space of differentials. This means that the  
differential 
\begin{equation*}
d\zeta:=\frac12 \sum_{s=1}^{2g+2}\epsilon_s\frac{d\eta_M}{x-e_s}+
\sum_{j=0}^{g-1}\beta_j\frac{x^jdx}w 
\end{equation*}
determined by the tangent vector satisfies the equations
\begin{eqnarray*}
 \int_{C_{s}} d\zeta &=& 0, \text{ for all } s=1,\dots,2g, \text{ and}\\
\left(\int_{D_{j}}+\frac12\int_{D_0}\right) d\zeta &=& 0, \text{ for all } j=1,\dots,g+2\,.
\end{eqnarray*}
All the periods of the $d\zeta$, both polar and cyclic, vanish and therefore its 
integral is a single valued function on the surface $M({\sf E_0})$:
\be
\zeta(P):=\frac12\left(\int_{P_0}^P+\int_{JP_0}^P\right)d\zeta\,, \text{ with } \zeta(P_0)=p_0.
\ee
The differential $d\zeta$ is odd with respect to the hyperelliptic involution $J$ and so is its integral 
$\zeta(P)$ for the chosen constant of integration. The only possible singularities of the meromorphic function 
$\zeta(P)$ are simple
poles at the branchpoints of $M$, whose number is not greater than 
$2g+2$. It is strictly less than the $2g+4$ zeros of $\zeta(P)$, which cover all the endpoints $x=p_s$ of the integration paths $D_s$ in formulas above.
Hence $d\zeta$ and therefore the annihilating tangent vector vanish.

We conclude that the set locally defined by fixing the values of all periods of the differential 
$d\eta({\sf B,E})$ is a smooth complex analytic manifold of dimension $g+2$ in the fibration over the space 
$\tilde{\cal H}_g$. It remains to show that it does not spoil under the projection to the base 
$\tilde{\cal H}_g$. If the isoperiodic manifold had two points gluing under the 
projection or it had a vertical tangent, that would mean the existence of a non zero holomorphic differential with vanishing periods. 
The latter is prohibited by the Riemann bilinear relations.
\end{proof}

Note that isoperiodic (or Pell-Abel) manifolds $\tilde{\mathscr{A}}_{g}^{n}$ are invariant under the action 
the $1$-dimensional affine group $ {\sf E}\to a{\sf E}+b$ with $(a,b)\in\CC^*\times\CC$. 
Indeed, this transformation does not change the conformal structure on the  Riemann surface with the marked point 
at infinity. Hence, the distinguished differential and all its periods survive under this map.  

\begin{cor}
 The quotient $\mathscr{A}_{g}^{n}$ of $\tilde{\mathscr{A}}_{g}^{n}$ by the action of the affine group is a smooth orbifold of complex dimension $g$.
\end{cor}

\begin{proof}
 This follows directly from Theorem~\ref{PAmanifold} and the fact that the action on the 
  affine group on any  set of $2g+2$ points in the plane has finite stabilizer.
\end{proof}

\section{Pictorial representation}
 \label{sec:graphcalcul}
In this section we introduce a pictorial representation for the description of the moduli space of hyperelliptic Riemann surfaces $M_{\infty}$ carrying a couple of marked points $\infty_\pm$ conjugated by the hyperelliptic involution. To such a  Riemann surface we  associate the planar graph whose edges are critical
leaves of the vertical and horizontal foliations of the distinguished quadratic differential $(d\eta_M)^2$ introduced in Section~\ref{sec:solv}. We will totally characterize such graphs
and each of them will come from a unique, up to the action of the affine group, pointed  Riemann surface~$M_{\infty}$. 

Originally this graphic language was designed in \cite{B03, Bbook} for the theory of real extremal polynomials, where the problem of 
 Riemann surfaces deformation with control of the periods exists too. It turned out to be very useful in the investigation of the global periods map, in particular for its 
image \cite{B03} and study of the topology of its fibers~\cite{B19}.

\subsection{Global width function}
Let $M_{\infty}$ be a hyperelliptic Riemann surface with marked points $\infty_\pm$ and $d\eta$ be its distinguished differential.
Given a branch point $e\in{\sf E}$, we define the \emph{width function} $W: \CC \to \RR_{+}$  by
\begin{equation}
 \label{W}
W(x)=\left| Re\int_{(e,0)}^{(x,w)} d\eta\right| \,.
\end{equation}
One immediately checks that the normalization conditions of the distinguished differential
imply that the width function satisfies the following properties:
\begin{enumerate}[i)]
\item $W$ is a well-defined single valued function on the plane.
\item $W$ is harmonic outside its zero set
$\Gamma_\vert:=\{x\in\mathbb{C}: W(x)=0\}$.
\item $W$ has a logarithmic pole at infinity.
\item $W$ vanishes at each branch point $e'\in{\sf E}$.
\end{enumerate}

We only comment on the Property~iv). Since $d\eta$ is odd with respect to hyperelliptic involution, the value $W(e')$ is equal to one half of the modulus of the real part of some period of $d\eta$. Since  all its periods are purely imaginary,  this gives iv). Moreover, this implies that the width function is independent on the
choice of the branch point $e$ as initial point of integration.

\subsection{Construction of the associated graph $\Gamma(M)$.}

Recall that a quadratic differential induces a vertical and a horizontal foliations (see \cite{strebel} for a detailed discussion).
The level lines of the width function  are the trajectories of the vertical foliation of the
distinguished quadratic differential $(d\eta)^2$, while the steepest descent lines of $W(x)$ are its horizontal trajectories.

To any  Riemann surface $M$ we associate a weighted \emph{planar}
graph $\Gamma=\Gamma(M)$ which is a union of a 'vertical' subgraph $\Gamma_\vert$ and
a 'horizontal' subgraph $\Gamma_\hor$. The precise definition is given below and 
examples of such graphs are given in Figure~\ref{Graph123}.
\begin{defn}
Let $M$ be the hyperelliptic Riemann surface given by Equation~\eqref{M}. Its \emph{associated graph}  $\Gamma(M)$ is the weighted planar graph constructed as follows:
\begin{itemize}
\item  The vertical edges are the unoriented arcs of the  zero set of $W(x)$ (they are
segments of the vertical foliation of  $(d\eta)^2$). 
\item The horizontal edges are the segments of the horizontal foliation
of $(d\eta)^2$ 
connecting saddle points of the function $W$ to the zero level set of $W$
(and may occasionally hit other saddle points on its path). The horizontal edges are oriented with
respect to the growth of $W(x)$.
\item  The vertices of the graph $\Gamma$ are the union of the finite points of
the divisor of $(d\eta)^2$ and the points in
$\Gamma_\vert\cap\Gamma_\hor$, i.e. projections of the saddle points of
$W$ to its zero set along the horizontal leaves.
\item Each edge $R$ of the graph is equipped with its
length $h(R)$ in the metric $ds:=|d\eta|$ induced by~$(d\eta)^2$.
\end{itemize}
\end{defn}

\begin{conv}\label{conv:grphrep}
 In the figures, we draw the vertical edges of the canonical graph with double lines.
 The horizontal edges are represented by  single line with an arrow showing their orientation. The weight of a vertical edge $R$ is denoted by $h(R)$.
We usually do not put the values of the horizontal weights on the figures.
\end{conv}

  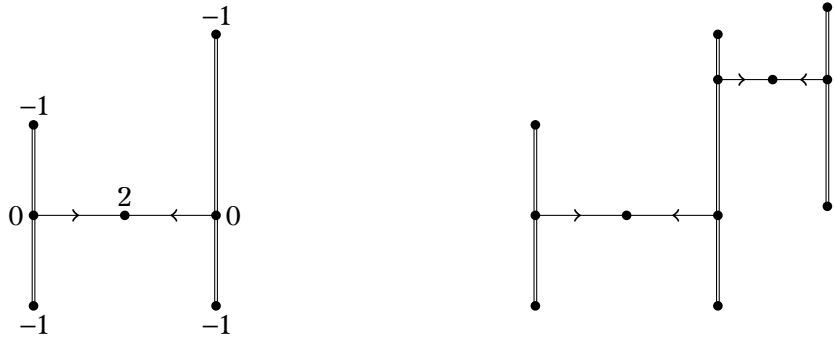
\begin{figure}[ht]
 \centering
\begin{tikzpicture}[scale=1.2,decoration={
    markings,
    mark=at position 0.5 with {\arrow[thick]{>}}}]

    \begin{scope}[xshift=-4.5cm]
     \coordinate (p0) at (0,-1);\node[below] at (p0) {$-1$};
\coordinate (p1) at (0,1);\node[above] at (p1) {$-1$};
\coordinate (p2) at (2,-1);\node[below] at (p2) {$-1$};
\coordinate (p3) at (2,2);\node[above] at (p3) {$-1$};
\coordinate (q1) at (1,0);\node[above] at (q1) {$2$};
\coordinate (q0) at (0,0);\node[left] at (q0) {$0$};
\coordinate (q2) at (2,0);\node[right] at (q2) {$0$};

\draw[postaction={decorate}] (q0) -- (q1) coordinate[pos=.5] (s1);\draw[postaction={decorate}] (q2) -- (q1) coordinate[pos=.5] (s2);
\draw[double] (p0) --(p1)coordinate[pos=.5] (r1);\draw[double] (p2) --(p3)coordinate[pos=.5] (r2);

      \foreach \i in {0,1,...,3}
   \fill (p\i)  circle (1.5pt);
   \foreach \i in {0,1,2}
  \fill (q\i)  circle (1.5pt);
    \end{scope}
    
        \begin{scope}[xshift=1cm]
     \coordinate (p0) at (0,-1);
\coordinate (p1) at (0,1);
\coordinate (p2) at (2,-1);
\coordinate (p3) at (2,2);
\coordinate (q1) at (1,0);
\coordinate (q0) at (0,0);
\coordinate (q2) at (2,0);

\draw[postaction={decorate}] (q0) -- (q1) coordinate[pos=.5] (s1);\draw[postaction={decorate}] (q2) -- (q1) coordinate[pos=.5] (s2);
\draw[double] (p0) --(p1)coordinate[pos=.5] (r1);\draw[double] (p2) --(p3)coordinate[pos=.5] (r2);

      \foreach \i in {0,1,...,3}
   \fill (p\i)  circle (1.5pt);
   \foreach \i in {0,1,2}
  \fill (q\i)  circle (1.5pt);
  
    \coordinate (r0) at (2,1.5);
\coordinate (r1) at (2.6,1.5);
\coordinate (r2) at (3.2,1.5);
\coordinate (r3) at (3.2,2.3);
\coordinate (r4) at (3.2,.1);

\draw[postaction={decorate}] (r0) -- (r1);\draw[postaction={decorate}] (r2) -- (r1);
\draw[double] (r2) --(r3);\draw[double] (r2) --(r4);
      \foreach \i in {0,1,2,3,4}
   \fill (r\i)  circle (1.5pt);
    \end{scope}

\end{tikzpicture}
\caption{Typical graphs associated to  Riemann surfaces of genera $1$ and $2$ are shown without their weights. For every vertex $V$ of the first graph, the value of $\ord(V)$ is given.}  \label{Graph123}
\end{figure}

From the local behaviour of the trajectories one immediately checks that
for any vertex $V\in\Gamma$ its multiplicity in the divisor of $(d\eta)^2$ 
is given by
\be
\ord (V):= d_\vert(V)+2d_{in}(V)-2\,,
\label{ordV}
\ee
where $d_\vert$ is the degree of the vertex with respect to the
vertical edges and $d_{in}$ is the number of incoming horizontal
edges. The branch points of $M$ correspond to the vertices~$V$  with the odd value of $\ord(V)$ and automatically lie on the vertical part of the graph~$\Gamma$. One can check it for the graphs represented in Figure~\ref{Graph123}.

\subsection{Admissible graphs}
The graphs $\Gamma(M)$ associated to hyperelliptic Riemann surfaces by the previous construction can be described in an axiomatic way.  
There are five conditions, three on its topology (T) and two on its weights (W).
\begin{thrm}\label{thm:five}
A weighted planar graph $\Gamma$ considered as a topological object (up to isopoty of the plane) 
is associated to a hyperelliptic Riemann surface $M$  if and only if the following five conditions are satisfied.
\begin{itemize}
 \item[(T1)]    The graph $\Gamma$ is a tree.
 \item[(T2)]   The horizontal edges leaving the same vertex are separated by a vertical or an incoming edge.
 \item[(T3)]  If $\ord (V)=0$ then $V\in\Gamma_\hor\cap\Gamma_\vert$. 
 \item[(W1)] The width function increases along oriented edges and $W(V)=0$ if $V$ lies on the vertical part of the graph. 
 \item[(W2)] The weights of vertical edges are positive and their total sum is $\pi$. 
\end{itemize}
Given a graph $\Gamma$ satisfying all five conditions, the  Riemann surface $M$ whose associated graph is $\Gamma$ is unique up to the action
of the linear maps $\Aff(1,\CC)$ on the branching set $\sf E$.
\end{thrm}

\begin{rmk}
These conditions imply some basic restrictions on the graphs $\Gamma(M)$.
For instance, there are no pendent horizontal edges like~~ 
\begin{picture}(7,3)
\put(0,1){\circle*{1}}
\put(0,1){\vector(1,0){5}}
\end{picture} 
~or~ 
\begin{picture}(7,3)
\put(0,1){\circle*{1}}
\put(5,1){\vector(-1,0){5}}
\end{picture}. 
The first case is prohibited by (T2),  while the second one is prohibited by  (T3). 
\end{rmk}

\begin{proof}
We give a sketch of the proof for completeness and the reader can look at  \cite{B03, Bbook} for a more detailed description. 

\paragraph{Constraints on associated graph.}
We say a few words about the genesis of properties (T1), (T2), (W2). The properties (T3) and (W1) directly follow from the definition of the graph $\Gamma$.
\smallskip
\par
\paragraph{Property (T1).} Suppose that the complement $\CC \setminus\Gamma(M)$ of the graph 
is not connected.  Let us calculate the Dirichlet integral of 
the width function in a bounded component $\Omega$ of the complement
by means of Green's formula:
$$
\int_\Omega |grad~W(x)|^2 d\Omega = 
\int_{\partial\Omega} W(x)\frac{\partial W}{\partial n} ~ds \,. 
$$
The function $W$ vanishes on the vertical parts of the boundary while its normal derivative vanishes at the horizontal parts of ${\partial\Omega}$.
This would imply that $W$ is constant. Now suppose that the graph has several components. Summing up the values 
of $\ord(V)$ over all its vertices, we get by Equation~\eqref{ordV} that 
$$
2\sharp\{\text{vertical edges}\}+2\sharp\{\text{horizontal edges}\}-2\sharp\{\text{vertices}\}=
-2\sharp\{\text{components of }\Gamma\} \,.
$$ 
This value equals to the  degree of the divisor of $(d\eta_M)^2$ on the sphere (i.e. -4) plus the order of its pole at infinity (i.e. 2).  
Hence, the graph $\Gamma$ has just one component and it is a single tree.
\smallskip
\par
\paragraph{Property (T2).}  Let $V$ be a vertex of $\Gamma$ such that $W(V)>0$.  This is a saddle point of the width function, the meeting point of 
several alternating ``ridges'' and ``valleys''. A horizontal edge comes into $V$ from each valley by definition. The outgoing edge (if any) goes along the ridge, so any two of them are separated. Same is true for $W(V)=0$ with vertical edges coming from each ``valley''.
\smallskip
\par
\paragraph{Property (W2).}  The integral of $(d\eta)^{2}$ along the boundary of the plane cut along $\Gamma_\vert$ equals~$2i$ times the sum of the weights of all vertical edges.   
The integration path may be contracted to the path encompassing the pole at infinity, hence by the residue theorem is $2i\pi$. 
\smallskip
\par
\paragraph{From the graph to the  Riemann surface.}
The Riemann surface $M$ may be glued from a finite number of stripes
in a way determined by combinatorics and weights of the graph. 
Below we briefly describe the procedure.

\begin{figure}
 \centering
\begin{tikzpicture}[scale=1.2,decoration={
    markings,
    mark=at position 0.5 with {\arrow[very thick]{>}}}]

    \coordinate (p0) at (0,-1);
\coordinate (p1) at (0,1);
\coordinate (p2) at (2,-1);
\coordinate (p3) at (2,2);
\coordinate (q1) at (1,0);
\coordinate (q0) at (0,0);
\coordinate (q2) at (2,0);

\draw[postaction={decorate}] (q0) -- (q1) coordinate[pos=.5] (s1);\draw[postaction={decorate}] (q2) -- (q1) coordinate[pos=.5] (s2);
\draw[double] (p0) --(p1)coordinate[pos=.5] (r1);\draw[double] (p2) --(p3)coordinate[pos=.5] (r2);

      \foreach \i in {0,1,...,3}
   \fill (p\i)  circle (1.5pt);
   \foreach \i in {0,1,2}
  \fill (q\i)  circle (1.5pt);

\draw[postaction={decorate},dashed] (p0) -- ++(0,-1);
\draw[postaction={decorate},dashed] (p1) -- ++(0,2);
\draw[postaction={decorate},dashed] (p2) -- ++(0,-1);
\draw[postaction={decorate},dashed] (p3) -- ++(0,1);
\draw[postaction={decorate},dashed] (q0) -- ++(-1.1,0);
\draw[postaction={decorate},dashed] (q1) -- ++(0,3);
\draw[postaction={decorate},dashed] (q1) -- ++(0,-2);
\draw[postaction={decorate},dashed] (q2) -- ++(1.1,0);

\foreach \i in {0,1,...,3}
   \fill (p\i)  circle (1.5pt);
  \fill (q1)  circle (1.5pt);

     \begin{scope}[xshift=5.5cm]
     \coordinate (p0) at (0,-1);
\coordinate (p1) at (0,1);
\coordinate (p2) at (2,-1);
\coordinate (p3) at (2,2);
\coordinate (q1) at (1,0);
\coordinate (q0) at (0,0);
\coordinate (q2) at (2,0);

\draw[postaction={decorate}] (q0) -- (q1) coordinate[pos=.5] (s1);\draw[postaction={decorate}] (q2) -- (q1) coordinate[pos=.5] (s2);
\draw[double] (p0) --(p1)coordinate[pos=.5] (r1);\draw[double] (p2) --(p3)coordinate[pos=.5] (r2);

      \foreach \i in {0,1,...,3}
   \fill (p\i)  circle (1.5pt);
   \foreach \i in {0,1,2}
  \fill (q\i)  circle (1.5pt);
  
    \coordinate (r0) at (2,1.5);
\coordinate (r1) at (2.6,1.5);
\coordinate (r2) at (3.2,1.5);
\coordinate (r3) at (3.2,2.3);
\coordinate (r4) at (3.2,.1);

\draw[postaction={decorate}] (r0) -- (r1);\draw[postaction={decorate}] (r2) -- (r1);
\draw[double] (r2) --(r3);\draw[double] (r2) --(r4);
      \foreach \i in {0,1,2,3,4}
   \fill (r\i)  circle (1.5pt);
   
   \draw[postaction={decorate},dashed] (p0) -- ++(0,-1);
\draw[postaction={decorate},dashed] (p1) -- ++(0,2);
\draw[postaction={decorate},dashed] (p2) -- ++(0,-1);
\draw[postaction={decorate},dashed] (p3) -- ++(0,1);
\draw[postaction={decorate},dashed] (q0) -- ++(-1.2,0);
\draw[postaction={decorate},dashed] (q1) -- ++(0,3);
\draw[postaction={decorate},dashed] (q1) -- ++(0,-2);
 \draw[postaction={decorate},dashed] (q2)   .. controls ++(0:.4) and ++(90:.5) ..  (2.3,-2.01);
 
 \draw[postaction={decorate},dashed] (r1) -- ++(0,1.45);
  \draw[postaction={decorate},dashed] (r1) -- ++(0,-3.55);
   \draw[postaction={decorate},dashed] (r0) .. controls ++(180:.5) and ++(-90:.5) ..  (1.55,3);
     \draw[postaction={decorate},dashed] (r2) -- ++(.8,0);
          \draw[postaction={decorate},dashed] (r3) -- ++(0,.6);
                    \draw[postaction={decorate},dashed] (r4) -- ++(0,-2.15);
    \end{scope}

\end{tikzpicture}
 \caption{The extensions of the graphs of figure~\ref{Graph123}.}
 \label{ExtGraph}
\end{figure}
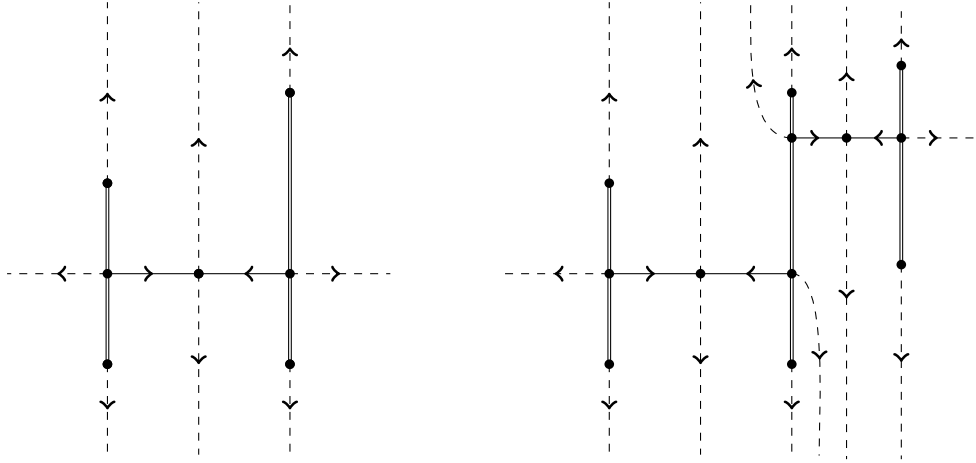

Given a planar graph satisfying the above five conditions, we extend it by drawing $d_\vert(V)-d_{out}(V)+d_{in}(V)\ge0$ outgoing 
horizontal arcs which connect each vertex $V$ to infinity and are disjoint except possibly at their endpoints. For each vertex 
we require that all the outgoing edges of this 
\emph{extended graph} $\Ext\Gamma$, old and new, alternate with the incident edges of other types: incoming or vertical, so that the graph $\Ext\Gamma$ satisfies Property~(T2). Since the original graph is a tree, the extended graph is unique up to isotopy of the plane. Typical  examples for $g=1$ and $g=2$ are given at Figure~\ref{ExtGraph}. 

From the topological viewpoint all the components of the complement to the extended graph in the plane have the same structure. 
They are $2$-cells bounded by exactly one vertical edge $R$ and two finite chains of horizontal edges attached to the endpoints of~$R$, all pointing away from 
the vertical edge and meeting at infinity.  For each cell we denote by $h(R)$ the weight of the corresponding vertical edge and define the half-strip for $h=h(R)$
by
$$
\Sigma(h)= \{\eta\in\mathbb{C}: \Reel \eta>0 \text{ and } 0<\Imagin \eta<h\}\,.
$$
We glue these $2\sharp\{\text{vertical edges}\}$ half-strips by translations along the horizontal edges and rotation of angle $\pi$ along the vertical edges as indicated by
the graph $\Ext\Gamma$. This flat  structure on the Riemann sphere has $2g+2$ singularities of odd order and is well defined up to the action of the affine group. 
The  Riemann surface $M$ is defined to be the double cover ramified at these points. The distinguished quadratic differential on the Riemann sphere is the one whose flat structure has been just defined.
\end{proof}

\begin{rmk}
1) The axiomatic description of the graphs $\Gamma$ appearing as associated graphs of  Riemann surfaces, including the five constraints
(T1,T2,T3,W1,W2) and the realization theorem, were first established for the 
Riemann surfaces admitting an anticonformal  involution (i.e. reflection) in  \cite{Bbook,B03}. The purely complex case is somewhat simpler as we should not keep in mind this mirror symmetry and arising additional topolological invariants, splitting of homology, etc. 

2) An interesting enumerative problem related to associated graphs arises: compute the number of (stable) combinatorial graphs $\Gamma$ associated to the  Riemann surfaces $M$ of genus $g$. Same holds for real curves with given genus and the number of real ovals.   
\end{rmk}

\subsection{Period mapping in terms of graphs}
\label{sec:graphtoperiod}

In this section we explain how to compute the periods of the distinguished differential from a graph satisfying the conditions of Theorem~\ref{thm:five}.

\subsubsection{Homology basis associated to a graph}\label{sec:homoassgraph}
Given a graph $\Gamma=\Gamma(M)$, we associate a set of $2g+2$ cycles on the twice punctured surface $M=M_{\infty}\setminus\infty_\pm$ which generate its integer 
homology group $H_1(M, \mathbb{Z})=\mathbb{Z}^{2g+1}$. This set is unique if all the branchpoints are pendent (degree one) vertices of the graph which is a generic case, see example on Figure~\ref{HomBasis}.

We denote the complex plane cut along the vertical part of the graph by 
$M^+:=\mathbb{C}\setminus\Gamma_\vert$.  The Riemann surface $M$ is obtained by gluing two copies of $M^+$ along the cuts in a criss-cross manner.

Suppose that we travel counterclockwise  along the boundary of the plane cut along the whole graph $\Gamma$. We meet each branching point $e$ exactly once, provided each of those are hanging vertices of the tree.
So all the branchpoints become cyclically ordered e.g. 
$e_1,e_2,\dots,e_{2g+1},e_{2g+2},e_{2g+3}=e_1$. In case there are interior branching points, 
some points $e\in{\sf E}$ will be listed more than once and we eliminate all duplicates in an arbitrary way.  
We again get a cyclic order of all the branchpoints, however not a unique one.

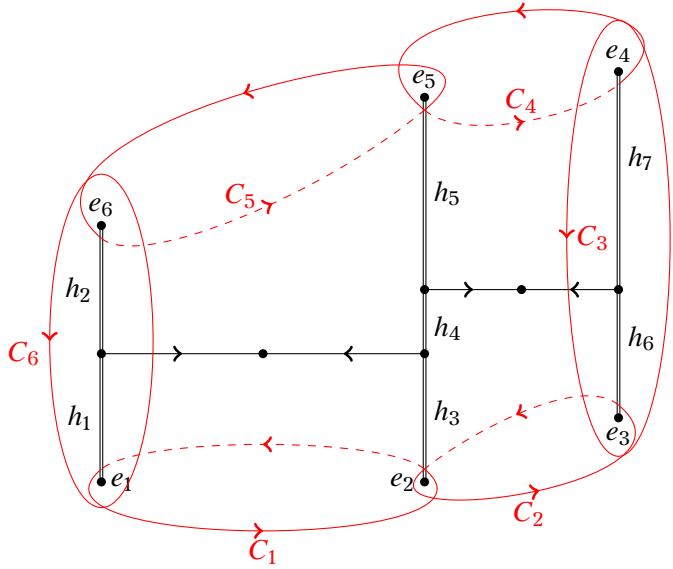
\begin{figure}[ht]
 \centering
\begin{tikzpicture}[scale=1.7,decoration={
    markings,
    mark=at position 0.5 with {\arrow[very thick]{>}}}]
     \coordinate (p0) at (-.5,-1);
\coordinate (p1) at (-.5,1);
\coordinate (p2) at (2,-1);
\coordinate (p3) at (2,2);
\coordinate (p4) at (3.5,-.5);
\coordinate (p5) at (3.5,2.2);
\coordinate (q1) at (.75,0);
\coordinate (q0) at (-.5,0);
\coordinate (q2) at (2,0);
\coordinate (q3) at (2,.5);
\coordinate (q4) at (2.75,.5);
\coordinate (q5) at (3.5,.5);

\draw[postaction={decorate}] (q0) -- (q1) coordinate[pos=.5] (s1);\draw[postaction={decorate}] (q2) -- (q1) coordinate[pos=.5] (s2);
\draw[double] (p0) --(p1)coordinate[pos=.5] (r1) coordinate[pos=.25] (h1) coordinate[pos=.75] (h2);\draw[double] (p2) --(p3)coordinate[pos=.5] (r2) coordinate[pos=.18] (h3)  coordinate[pos=.4] (h4) coordinate[pos=.75] (h5);
\draw[postaction={decorate}] (q3) -- (q4) coordinate[pos=.5] (s3);\draw[postaction={decorate}] (q5) -- (q4) coordinate[pos=.5] (s4);
\draw[double] (p4) --(p5)coordinate[pos=.5] (r3) coordinate[pos=.23] (h6) coordinate[pos=.75] (h7);

\foreach \i in {0,1,...,5}
   \fill (p\i)  circle (1pt);
   \foreach \i in {0,1,...,5}
  \fill (q\i)  circle (1pt);

\node[right] at (p0) {$e_{1}$};
\node[above] at (p1) {$e_{6}$};
\node[left] at (p2) {$e_{2}$};
\node[above] at (p3) {$e_{5}$};
\node[below] at (p4) {$e_{3}$};
\node[above] at (p5) {$e_{4}$};

   \draw[postaction={decorate},red,dashed] (-.5,.9) .. controls ++(-40:.4) and ++(-140:1) .. node[above] {$C_{5}$} (2,1.9);
          \draw[postaction={decorate},red] (2,1.9) .. controls ++(40:1.4) and ++(140:1.4) .. (-.5,.9);
 \draw[postaction={decorate},red,dashed] (2,1.9) .. controls ++(-40:.4) and ++(-140:.4) .. node[above] {$C_{4}$} (3.5,2.1);
          \draw[postaction={decorate},red] (3.5,2.1) .. controls ++(40:1.4) and ++(140:1.4) .. (2,1.9);

\draw[postaction={decorate},red] (-.5,-.9) .. controls ++(-140:1) and ++(-40:1)  .. node[below] {$C_{1}$} (2,-.9) ;
\draw[postaction={decorate},red,dashed] (2,-.9) .. controls ++(140:.4) and ++(40:.4) .. (-.5,-.9);
\draw[postaction={decorate},red] (2,-.9) .. controls ++(-140:.8) and ++(-40:1) .. node[below] {$C_{2}$} (3.5,-.4);          \draw[postaction={decorate},red,dashed] (3.5,-.4) .. controls ++(140:.4) and ++(40:.4) .. (2,-.9);

\draw[postaction={decorate},red] (3.5,.9) ellipse (.4cm and 1.7cm) node[left] {$C_{3}$};
\draw[postaction={decorate},red] (-.5,0.1) ellipse (.4cm and 1.3cm);
\node[red]  at (-1.1,0) {$C_{6}$};

\node[left] at (h1) {$h_{1}$};
\node[left] at (h2) {$h_{2}$};
\node[right] at (h3) {$h_{3}$};
\node[right] at (h4) {$h_{4}$};
\node[right] at (h5) {$h_{5}$};
\node[right] at (h6) {$h_{6}$};
\node[right] at (h7) {$h_{7}$};
\end{tikzpicture}
 \caption{Generators of the first homology group of the genus  $g=2$ Riemann surface associated to the generic graph  $\Gamma$.} \label{HomBasis}
\end{figure}

For $ j=1,\dots,2g+2$,  let $c_j$ be any simple arc connecting point $e_j$ to $e_{j+1}$ and disjoint from the graph $\Gamma$ except for its ends. We draw this arc on $M^+$ and then $C_j:=(\Id-J)c_j$ is a closed loop on the surface $M$. Those $2g+2$ loops are represented in Figure~\ref{HomBasis}. They are linearly dependent: both sums of the loops with even/odd indexes are equal to the same loop encircling clockwise the puncture $\infty_+$. There are no other relations between them:
\begin{lmm}
The cycles $C_1,C_2,\dots , C_{2g+1}$ make up a basis of the lattice  $H_1(M,\mathbb{Z})$.
\end{lmm}

\begin{proof}
  For every $j=1,\dots,2g+2$ consider the relative cycles $D_{j}$ in the relative  homology group $H_1(M_{\infty},\{\infty_\pm\}, \mathbb{Z})$ given by $D_j:=(\Id-J)d_j$
 where $d_j$ is any simple arc connecting branch point $e_j$ to $\infty_+$ and disjoint from the graph except for its starting point.
There is a pairing between the above two homology groups given by the intersection index.
We compute that $D_s\circ C_j= 1 $ if $s=j$ or $s=j+1$ and  is equal to $0$ for all other indexes. The determinant of the intersection matrix $||D_s\circ C_j||_{s,j=1}^{2g+1}$ equals to $1$.
\end{proof}

\subsubsection{Period mapping for the associated homology basis}\label{sec:perhombas}
Given an admissible graph $\Gamma$, we can calculate the periods of the distinguished differential $d\eta$ along the basic cycles $C_{j}$ introduced in Section~\ref{sec:homoassgraph}. This  differential may be reconstructed from the width function as $d\eta=2\partial W(z)$ on the top sheet $M^+$. On the other sheet it just has the opposite sign.

\begin{lmm} \label{lm:period}
The period of the distinguished differential $d\eta$ along the cycle $C_{j}$ is
\be
\int_{C_j}d\eta=2i\sum_{e_j<R<e_{j+1}} h(R) \,,
\label{period}
\ee
where the summation is taken over all vertical edges $R$ of $\Gamma$ that appear 
when travelling counterclockwise from $e_j$ to $e_{j+1}$ along the bank of~$\Gamma$.
\end{lmm}

\begin{proof}
Let  $H(z)$ be the harmonic conjugate to the width function $W(z)$. It is a multivalued  function in  the complement of the graph $\Gamma$: going around the graph (or equivalently, the  infinity) adds $\pm2\pi$ to the initial value of $H(z)$. We have a chain of equalities:
\be
\int_{C_j}d\eta=
2\int_{c_j}d\eta= 
2\int_{c_j}d(W+iH)= 
2i\int_{c_j}dH \,.
\ee 
To obtain the last equality we used that the width function $W$ vanishes at all the branch points $e$, where the path $c_{j}$ starts and ends. Continuing the last equality:
\be
\int_{C_j}d\eta
=2i\sum_{e_j<R<e_{j+1}} \int_R dH=
2i\sum_{e_j<R<e_{j+1}} h(R) \,.
\ee
Here we used Cauchy-Riemann equations 
$$
dH|_R=\frac{\partial W}{\partial n}dl\,,
$$ 
where $n$ is normal to the edge $R$
and $l$ is a length parameter on the edge. Hence $dH$ vanishes at the horizontal edges 
and is equal to the metric of the differential $|d\eta|$ on the vertical edges.
\end{proof}

\begin{ex}
 For the graph pictured in Figure~\ref{HomBasis}, the period of $d\eta$ along $C_{1}$ is $2i (h_{1}+h_{3})$, the period along  $C_{2}$ is $2i (h_{3}+h_{4}+h_{6})$ and the period 
 along $C_1+C_3+C_5$ equals to $2i (h_1+h_3+h_6+h_7+h_5+h_4+h_2) =2\pi i$
 according to normalization (W2).
\end{ex}

\subsection{Local isoperiodic deformations}\label{Sec:LocalDeform}
We  know  from Theorem~\ref{PAmanifold} that
fixing the values of periods of the distinguished differential locally 
define a complex $(g+2)$-dimensional submanifold, such as the $\tilde{\mathscr{A}}_{g}^{n}$,  in the moduli space~$\tilde{\cal H}_g$. Two degrees of freedom on this manifold account for inessential affine motions of the branching divisor which do not change the complex structure.
The remaining~$g$ complex degrees of freedom on an isoperiodic manifold may be explained in terms of associated graphs.  For simplicity we define the isoperiodic deformations for the generic graph, general case will follow from continuity.

In the generic case the width function $W$ has exactly $g$ saddle points $V$, that is 
the double zeros of $(d\eta)^2$.  The vicinity of each of them in the graph $\Gamma$ has the appearance pictured in Figure~\ref{IsoDef}: the vertex $V$ is the meeting point of exactly two horizontal edges which go 
straight from two vertical components of the graph.
Each of two nearest neighbour nodes of~$V$ is incident to exactly two vertical edges. 
We label the weights of the four mentioned nearest 
to $V$ vertical edges cyclically by $h_1$, $h_2$, $h_3$ and $h_4$ as in Figure~\ref{IsoDef}.
The following two modifications of weights in the neighbourhood of the vertex~$V$
obviously do not change any period:
\be
W(V)\to W(V)+\delta W;
\qquad h_s \to h_s-(-1)^s\delta h, 
\qquad s=1,2,3,4,
\label{Isoperiodic}
\ee
with real increments $\delta W$, $\delta h$ small enough for the modified graph to obey admissibility conditions.

 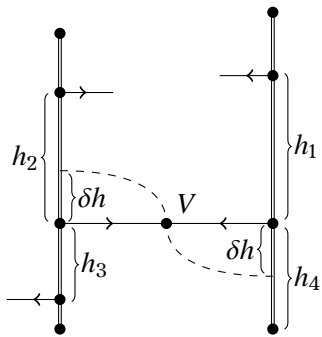
\begin{figure}[ht]
 \centering
\begin{tikzpicture}[scale=1.4,decoration={
    markings,
    mark=at position 0.5 with {\arrow[thick]{>}}}]

\coordinate (p0) at (0,-1);
\coordinate (p1) at (0,1.8);
\coordinate (p2) at (2,-1);
\coordinate (p3) at (2,2);
\coordinate (q1) at (1,0);
\coordinate (q0) at (0,0);
\coordinate (q2) at (2,0);

\draw[postaction={decorate}] (q0) -- (q1);\draw[postaction={decorate}] (q2) -- (q1);
\draw[double] (p0) --(p1)coordinate[pos=.1] (s1)coordinate[pos=.8] (s2);\draw[double] (p2) --(p3)coordinate[pos=.3] (s3)coordinate[pos=.2] (s3)coordinate[pos=.8] (s4);

\draw[postaction={decorate}] (s1) -- ++ (-.5,0);
\draw[postaction={decorate}] (s2) -- ++ (.5,0);
\draw[postaction={decorate}] (s4) -- ++ (-.5,0);

    \foreach \i in {1,2,4}
   \fill (s\i)  circle (1.5pt);

      \foreach \i in {0,1,...,3}
   \fill (p\i)  circle (1.5pt);
         \foreach \i in {0,1,2}
  \fill (q\i)  circle (1.5pt); \node[above right] at (q1) {$V$};
  
  \draw [decorate,decoration={brace}]	(-0.11,0.02) -- (-0.11,1.23) node [midway, left] {$h_{2}$};
  \draw [decorate,decoration={brace}]	(2.11,-0.04) -- (2.11,-1.0) node [midway, right] {$h_{4}$};
  \draw [decorate,decoration={brace}](0.11,-0.04)	 --(0.11,-.74)  node [midway, right] {$h_{3}$};
  \draw [decorate,decoration={brace}]	(2.11,1.42) -- (2.11,.04) node [midway, right] {$h_{1}$};

    \draw[dashed] (0,.5) .. controls ++(0:.7) and ++(90:.3) .. (q1) .. controls ++(-90:.3) and ++(180:.7) .. (2,-.5);
  \draw [decorate,decoration={brace}]	(0.08,.47) -- (0.08,0.02) node [midway, right] {$\delta h$};
  \draw [decorate,decoration={brace}]	(1.92,-.47) -- (1.92,-0.02) node [midway, left] {$\delta h$};
\end{tikzpicture}
\caption{Vicinity of a generic saddle point $V$. The horizontal segment of the graph deformed by $ h_s \to h_s-(-1)^s\delta h$ with positive $\delta h$ is pictured  in dashed.}
\label{IsoDef}
\end{figure}

We will use the deformations of this kind to bring the graph of a  Riemann surface~$M$
corresponding to Equation~\eqref{PA} admitting a  primitive solution of degree $n$ to a  standard form.

 \section{Isoperiodic deformation to graphs  of standard forms}\label{sec:upperbound}
  The original enumeration problem  essentially belongs to algebraic geometry, however the graph technology allows us to study it by efficient combinatorial methods. 
A similar approach is used  in the classification of the connected components of strata of abelian differentials \cite{KonZor}, intersection theory on moduli spaces \cite{Kon91,Kon92} and some other investigations.

In this section we first introduce two standard forms of the graphs $\Gamma(M)$ and secondly
present a combinatorial procedure for the isoperiodic deformation of a graph associated to a Pell-Abel equation with a primitive solution of degree $n$   to the standard form graph.

These standard forms may be chosen differently. For the upper bound of the number of 
connected components $a(g,n)$ we use the \emph{Two Bush} standard form. For the lower bound in Section~\ref{sec:braid} we will use the \emph{Linear} standard from. For sake of completeness we give an explicit isoperiodic transformation between both standard forms.

\subsection{Two standard forms of graphs}\label{sec:2stdforms}
Let $\Gamma$ be a  graph associated to a Pell-Abel equation with a primitive solution of degree~$n$. For convenience we rescale the weights of its vertical edges as follows:
\be \label{rescale}
\hbar(R):=nh(R)/\pi \,.
\ee
This rescaling allows us to work with integers instead of rational multiples of $\pi$. 
To distinguish between the normalizations, we continue to call the value $h(R)$ 
the weight of the (vertical) edge $R$, whereas we refer to $\hbar(R)$ as of its \emph{height}.
Note that the total height of the vertical component of a graph is equal to~$n$.

\begin{defn} 
The \emph{linear graph} $\Gamma(s,g,n)$ with integer parameters $g\geq1$, $n\geq g+1$ and $s=0,1,\dots,m^*:=$ $\min(g-1,n-g-1)$  is defined as follows. It has $g+1$ vertical segments connected at their endpoints by $g$ horizontal components so that the whole graph is embedded in a line, as represented in Figure~\ref{LinGraph}. The first $(g-s)$ vertical edges are of height $\hbar =1$, they are  followed by $s$ vertical edges of height $\hbar=2$ and finally the height of the last one is $\hbar=n-g-s$. The value of the width function at its $g$ saddle points is not specified as inessential.
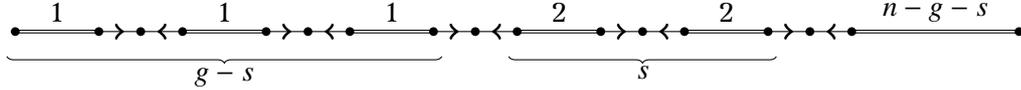
\begin{figure}[ht]
 \centering
\begin{tikzpicture}[scale=1.1,decoration={
    markings,
    mark=at position 0.6 with {\arrow[very thick]{>}}}]

    \foreach \i in {0,1,...,12}
     \coordinate (p\i) at (\i,0);
  \foreach \i in {0,1,...,12}
     \coordinate (q\i) at (\i+1/2,0);

\draw[postaction={decorate}] (p1) -- (q1);
\draw[postaction={decorate}] (p2) -- (q1);
\draw[postaction={decorate}] (p3) -- (q3);
\draw[postaction={decorate}] (p4) -- (q3);
\draw[postaction={decorate}] (p5) -- (q5);
\draw[postaction={decorate}] (p6) -- (q5);
\draw[postaction={decorate}] (p7) -- (q7);
\draw[postaction={decorate}] (p8) -- (q7);
\draw[postaction={decorate}] (p9) -- (q9);
\draw[postaction={decorate}] (p10) -- (q9);
\draw[double] (p0) --(p1)coordinate[pos=.5] (r1);
\draw[double] (p2) --(p3)coordinate[pos=.5] (r2);
\draw[double] (p4) --(p5)coordinate[pos=.5] (r3);
\draw[double] (p6) --(p7)coordinate[pos=.5] (r4);
\draw[double] (p8) --(p9)coordinate[pos=.5] (r5);
\draw[double] (p10) --(p12)coordinate[pos=.5] (r6);

\foreach \i in {0,1,2,3,4,5,6,7,8,9,10,12}
   \fill (p\i)  circle (1.5pt);
 \foreach \i in {1,3,5,7,9}
    \fill (q\i)  circle (1.5pt);
    \foreach \i in {1,2,3}
   \node[above] at (r\i) {$1$};
    \foreach \i in {4,5}
   \node[above] at (r\i) {$2$};
   \node[above] at (r6) {$n-g-s$};
   
   \draw [decorate,decoration={brace}]	(5.1,-.3) -- (-.1,-.3) node [midway, below] {$g-s$};
      \draw [decorate,decoration={brace}]	(9.1,-.3) -- (5.9,-.3) node [midway, below] {$s$};
\end{tikzpicture}
\caption{The linear graph $\Gamma(s,g,n)$ for $g=5$ and $s=2$.}
\label{LinGraph}
\end{figure}
\end{defn}

\begin{rmk}\label{rem:mstar}
The number $s$ of vertical edges of height $\hbar=2$ in the standard linear form cannot be too large when the degree $n$ is smaller than $2g$. Indeed, otherwise the last vertical edge will have zero or negative height. This is the reason  why $s$ is less or equal to $m^*:=\min(g-1,n-g-1)$. 
\end{rmk}

\begin{rmk} The linear graphs correspond to Riemann surfaces $M$ with only real branchpoints.
The solutions $P(x)$ of corresponding Pell-Abel equations are known as 
multiband Chebyshev polynomials, see \cite{B99, B03}. In this case, the heights $\hbar$ of the vertical segments correspond to the oscillation numbers of  the Chebyshev polynomial $P(x)$ on the bands. In general they can take arbitrary positive integer values
which sum up to $\deg P=n$.
\end{rmk}

Given the same set of parameters $(s,g,n)$ as for the  standard linear form
$\Gamma(s,g,n)$, we introduce the \emph{Two bush} standard form $\Gamma^*(s,g,n)$
built as follows:
\begin{defn}
The \emph{small bush} is a collection of $2(g-s)+2$ vertical edges, that we call \emph{twigs}, of equal height $\hbar =1/2$ all growing from the same root.  The \emph{large bush} is a similar starlike graph of $2s$ vertical edges of height $\hbar=1$.
The \emph{two bush graph} $\Gamma^*(s,g,n)$ is obtained by gluing  the root of the large bush and a vertical edge of height $\hbar=n-g-s-1$, called \emph{tail},  to a hanging  vertex of the small bush, in such a way that  the 
whole embedded graph admits reflection symmetry. Such graph is pictured in Figure~\ref{fig:2bush}.
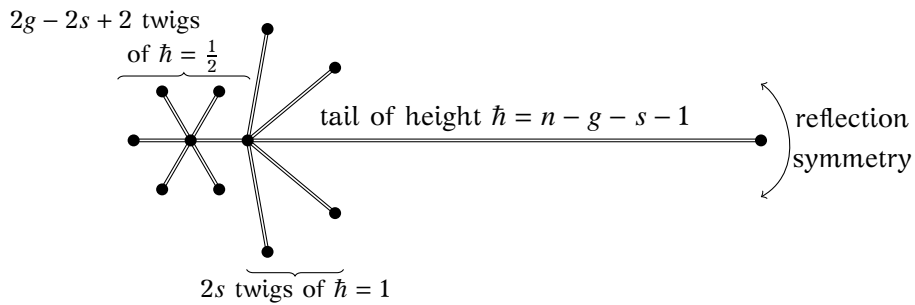
\begin{figure}[ht]
 \centering
\begin{tikzpicture}[scale=1.5,decoration={
    markings,
    mark=at position 0.5 with {\arrow[thick]{>}}}]
 \coordinate (p0) at (0,0);
 \foreach \i in {1,2}
\coordinate (p\i) at (40*\i:1);
\coordinate (p3) at (-80:1);
\coordinate (p4) at (-40:1);
\coordinate (q0) at (4.5,0);
          \foreach \i in {1,2,...,4}
   \draw[double] (p0) -- (p\i);
   \coordinate (p6) at (-1/2,0);
      \draw[double] (p0) -- (p6);
             \foreach \i in {1,2,3,4,5}
   \draw[double] (p6) -- ++(60*\i:.5) coordinate (q\i);
    \draw[double] (p0) --node[above] {tail of height $\hbar=n-g-s-1$} (q0);
    \foreach \i in {0,1,2,3,4,5}
   \fill (q\i)  circle (1.5pt);
       \foreach \i in {0,1,2,3,4,6}
   \fill (p\i)  circle (1.5pt);

  \draw [decorate,decoration={brace}]	(.84,-1.14) -- (-.01,-1.14) node [midway, below] {\small{$2s$ twigs of $\hbar= 1$}};
   \draw [decorate,decoration={brace}]	 (-1.14,.54)-- (.02,.54);
 \node at (-1.25,1.05) {\small{$2g-2s+2$ twigs}};
    \node at (-.65,.75) {\small{of $\hbar= \frac{1}{2}$}};

   \draw[<->] (4.5,.5)  .. controls ++(-30:.25) and ++ (30:.5) .. (4.5,-.5) ;
 \node at (5.3,.2) {reflection};
  \node at (5.3,-.2) {symmetry};
   \end{tikzpicture}
 \caption{The two Bush graph $\Gamma^*(s,g,n)$ for $g=4$ and $s=2$.} \label{fig:2bush}
\end{figure}
\end{defn}

Note that the tail disappears when $s=n-g-1$. In this case the root of the larger bush becomes a branch point.

 We will prove in Section~\ref{sec:prooflema}  that these two standard forms are related to each other in the following way:
\begin{lmm} 
\label{lmm:Line2Bush}
The two bush graph $\Gamma^*(s,g,n)$ and the linear graph $\Gamma(s,g,n)$  are joined by an isoperiodic deformation. 

The two bush graphs $\Gamma^*(s,g,n)$ and $\Gamma^*(s-1,g,n)$ with $s>0$  and  $s+g+n$ odd are joined by an isomorphic deformation. 
\end{lmm}

The main result of this section is the following.
\begin{thrm}\label{thm:isoperdef}
 Any graph $\Gamma$ corresponding to a Pell-Abel equation $P^2(x)-D(x)Q^2(x)=1$ with $\deg D=2g+2>2$ and
 admitting a  primitive solution of degree $n>g$ can be isoperiodically transformed into a two bush graph $\Gamma^*(s,g,n)$ for some  $s=0,1,\dots,m^*$, where $m^*:=\min(g-1, n-g-1)$.
\end{thrm}

 \begin{cor}
  The number of connected components $a(g,n)$ of $\mathscr{A}_{g}^{n}$ for $n>g$ and $g>0$
  is at most $[m/2]+1$ if $n+g$ is odd and at most  $[(m+1)/2]$ if $n+g$ is even, where $m=\min(g,n-g-1)$.
 \end{cor}

 \begin{proof}
 According to Lemma~\ref{lmm:Line2Bush}, the two bush graphs  $\Gamma^*(s,g,n)$ and $\Gamma^*(s-1,g,n)$ can be joint by an isoperiodic deformation if $s>0$  and  $s+g+n$ odd.
Now it suffices to count parameters to see that  the number of inequivalent two bush graphs is at most $[(m+1)/2]$ when $n+g$ is even and $[m/2]+1$ when $n+g$ is odd.
 \end{proof}

In order to prove Lemma~\ref{lmm:Line2Bush} and Theorem~\ref{thm:isoperdef} we give some preparatory material on the isoperiodic deformations.

\smallskip
\par
\subsection{Useful isoperiodic deformations}
\label{sec:prepadef}
In this preparatory section, we describe useful isoperiodic deformations of a graph associated to hyperelliptic Riemann surfaces $M_{\infty}$ with a pair of marked points  $\infty_\pm$ in involution.
\smallskip
\par
\paragraph{Rolling:}

Suppose that the graph $\Gamma$ has exactly two disjoint vertical components $\Gamma^1_\vert$ and~$\Gamma^2_\vert$. The latter are connected by the only horizontal component containing exactly two edges
meeting at the saddle point of the width function as left of Figure~\ref{fig:rotate}. We call such simple horizontal component a \emph{cord}.

The following deformation of $\Gamma$, called \emph{rolling} and pictured in Figure~\ref{fig:rotate}, is isoperiodic. The cord is fixed while both vertical components rotate as rigid bodies  in the same direction so that the meeting points of the cord with both vertical components 
move along the boundaries of $\Gamma^1_\vert$ and $\Gamma^2_\vert$ with equal speed. 
Alternatively, we keep one of the vertical components static, say $\Gamma^2_\vert$ which we now call the \emph{core} component. The cord goes around the core and drags
the other vertical component $\Gamma^1_\vert$ which at a time rotates with respect to the cord in the opposite direction so that the equality of velocities of the contact points again holds. 
Essentially this deformation is the same as that in Section~\ref{Sec:LocalDeform} however the parameter $\delta h$ of the deformation is no longer small.

\begin{figure}[ht]
 \centering
\begin{tikzpicture}[scale=1.2,decoration={
    markings,
    mark=at position 0.5 with {\arrow[thick]{>}}}]

\coordinate (p0) at (1,-1);
\coordinate (p1) at (1,1);
\coordinate (p2) at (2,-1);
\coordinate (p3) at (2,2);
\coordinate (q1) at (1.5,0);
\coordinate (q0) at (1,0);
\coordinate (q2) at (2,0);

\draw[postaction={decorate}] (q0) -- (q1);\draw[postaction={decorate}] (q2) -- (q1);
  \draw[dotted] (1,-.2) .. controls ++(0:.2) and ++(-170:.2) ..   (q1) .. controls ++(10:.2) and ++ (180:.2) .. (2,.2) ;
    \draw[dotted] (1,-.4) .. controls ++(0:.2) and ++(-160:.2) ..   (q1) .. controls ++(20:.2) and ++ (180:.2) .. (2,.4) ;
\draw[double] (p0) --(p1)coordinate[pos=.25] (s1)coordinate[pos=.7] (s2);\draw[double] (p2) --(p3)coordinate[pos=.3] (s3)coordinate[pos=.2] (s3)coordinate[pos=.8] (s4);

      \foreach \i in {0,1,...,3}
   \fill (p\i)  circle (1.5pt);
         \foreach \i in {0,1,2}
  \fill (q\i)  circle (1.5pt);

   \coordinate (r0) at (2,1.5);
\coordinate (r1) at (2.6,1.5);
\coordinate (r2) at (3.2,1.5);
\coordinate (r3) at (3.2,2.3);
\coordinate (r4) at (3.2,.1);

\draw[double] (r0) -- (r2);
\draw[double] (r2) --(r3);\draw[double] (r2) --(r4);
      \foreach \i in {0,2,3,4}
   \fill (r\i)  circle (1.5pt);

 \draw [decorate,decoration={brace}]	(1.1,-1.14) -- (.9,-1.14) node [midway, below] {$\Gamma^1_\vert$};
    \draw [decorate,decoration={brace}]	(3.44,-1.14) -- (1.74,-1.14) node [midway, below] {$\Gamma^2_\vert$};

  \draw[->] (3.7,.4) --  node[above]{rolling} (5.2,.4);

  \begin{scope}[xshift=4cm]

\coordinate (p1) at (1.2,2.9);
\coordinate (p2) at (2,-1);
\coordinate (p3) at (2,2);
\coordinate (q1) at (3.2,2.6);
\coordinate (q0) at (3.2,2.9);
   \coordinate (r0) at (2,1.5);
\coordinate (r1) at (2.6,1.5);
\coordinate (r2) at (3.2,1.5);
\coordinate (r3) at (3.2,2.3);
\coordinate (r4) at (3.2,.1);

\draw[postaction={decorate}] (q0) -- (q1);\draw[postaction={decorate}] (r3) -- (q1);
  \draw[dotted] (2.9,2.9) .. controls ++(-90:.2) and ++(150:.2) ..   (q1) .. controls ++(-30:.4) and ++ (0:.3) .. (3.2,2) ;
     \draw[dotted] (2.7,2.9) .. controls ++(60:.4) and ++(30:.8) ..   (q1) .. controls ++(-150:.4) and ++ (180:.4) .. (3.2,1.8) ;
\draw[double] (q0) --(p1)coordinate[pos=.5] (s1)coordinate[pos=.7] (s2);\draw[double] (p2) --(p3)coordinate[pos=.3] (s3)coordinate[pos=.2] (s3)coordinate[pos=.8] (s4);

      \foreach \i in {0,1,...,3}
   \fill (p\i)  circle (1.5pt);
         \foreach \i in {0,1,2}
  \fill (q\i)  circle (1.5pt);

\draw[double] (r0) -- (r2);
\draw[double] (r2) --(r3);\draw[double] (r2) --(r4);
      \foreach \i in {0,2,3,4}
   \fill (r\i)  circle (1.5pt);

 \draw [decorate,decoration={brace}]	(3.54,2.99) -- (3.54,2.79) node [midway, right] {$\Gamma^1_\vert$};
    \draw [decorate,decoration={brace}]	(3.54,2.34) -- (3.54,-1.14) node [midway, right] {$\Gamma^2_\vert$};
    
      \draw[->] (4.4,.4) --  node[above]{rolling} (5.9,.4);
  \end{scope}
  
    \begin{scope}[xshift=8.5cm]
\coordinate (p2) at (2,-1);
\coordinate (p3) at (2,2);
   \coordinate (r0) at (2,1.5);
\coordinate (r1) at (2.6,1.5);
\coordinate (r2) at (3.2,1.5);
\coordinate (r3) at (3.2,2.3);
\coordinate (r4) at (3.2,.1);

\coordinate (q1) at (3.5,1);
\coordinate (p1) at (3.8,1.7);
\coordinate (q0) at (3.8,-.3);
\coordinate (r5) at (3.2,1);
\coordinate (r6) at (3.8,1);

\draw[postaction={decorate}] (r5) -- (q1);\draw[postaction={decorate}] (r6) -- (q1);
\draw[double] (q0) --(p1)coordinate[pos=.5] (s1)coordinate[pos=.7] (s2);\draw[double] (p2) --(p3)coordinate[pos=.3] (s3)coordinate[pos=.2] (s3)coordinate[pos=.8] (s4);

      \foreach \i in {0,1,...,3}
   \fill (p\i)  circle (1.5pt);
         \foreach \i in {0,1,2}
  \fill (q\i)  circle (1.5pt);

\draw[double] (r0) -- (r2);
\draw[double] (r2) --(r3);\draw[double] (r2) --(r4);
      \foreach \i in {0,2,3,4,5,6}
   \fill (r\i)  circle (1.5pt);

 \draw [decorate,decoration={brace}]	(4,-1.14) -- (3.6,-1.14) node [midway, below] {$\Gamma^1_\vert$};
    \draw [decorate,decoration={brace}]	(3.44,-1.14) -- (1.74,-1.14) node [midway, below] {$\Gamma^2_\vert$};
  \end{scope}
\end{tikzpicture}
\caption{Rolling the vertical component $\Gamma^1_\vert$  around the core component $\Gamma^1_\vert$. The dotted lines show the intermediate positions of the chord.}
\label{fig:rotate}
\end{figure}
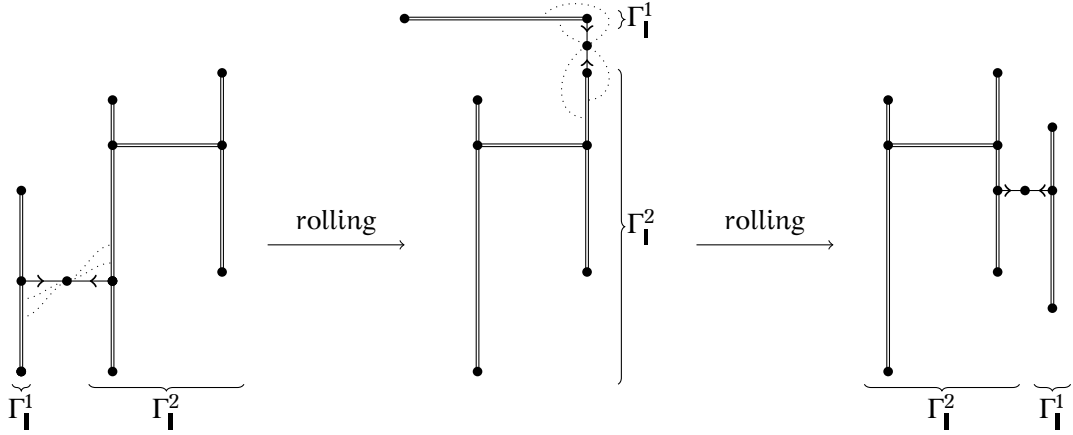

\begin{rmk}
The rolling of a pendent vertical component of the graph $\Gamma$ around the rest of the graph may be defined in a more general case.
For simplicity in this paper we do not use any deformations which lead to the collision of 
different horizontal components of the graph.
The collision of this type leads to a more deep change in combinatorial structure of the graph $\Gamma$, see e.g.  Chapter~4 of \cite{Bbook} and \cite{B23} for the analytical aspects of such collision.
\end{rmk}

\smallskip
\par
\paragraph{Attaching and Detaching:}
Given a graph $\Gamma$ as in the rolling procedure, the cord may be contracted when it reaches some point $V$ at the boundary of the core graph $\Gamma_\vert^2$ during rolling. This procedure is called \emph{attaching} of $\Gamma_\vert^1$ at the point $V$ on the core vertical graph. Note that if the cord connects two branch points, like in the middle of Figure~\ref{fig:rotate}, the procedure leads to a  nodal curve $M$ and it is prohibited. 
The inverse procedure of inserting a cord at a vertex $V$ of a vertical subgraph will be referred as a \emph{detaching}.

\smallskip
\par
\paragraph{Pumping:}
Given a graph $\Gamma$ with a pendent vertical segment $[V_1,V_2]=\Gamma_\vert^1$ and a core graph $\Gamma_\vert^2$. We can roll $\Gamma_\vert^1$ until  the cord passes through a branch point of $\Gamma_\vert^2$ as pictured left of Figure~\ref{Pump}. We can transfer a positive weight from $\Gamma_\vert^1$ to the core component by the following \emph{pumping} construction. We first contract the cord as in the middle of Figure~\ref{Pump} and 
then again insert it in another way as right of Figure~\ref{Pump}. 

Note that pumping mass is impossible if the cord simultaneously passes through two branch points: one on the pendent vertical segment and the other on the core graph, see in the middle of Figure~\ref {fig:rotate}. As observed before, in this case contracting the cord brings us to a nodal curve.

 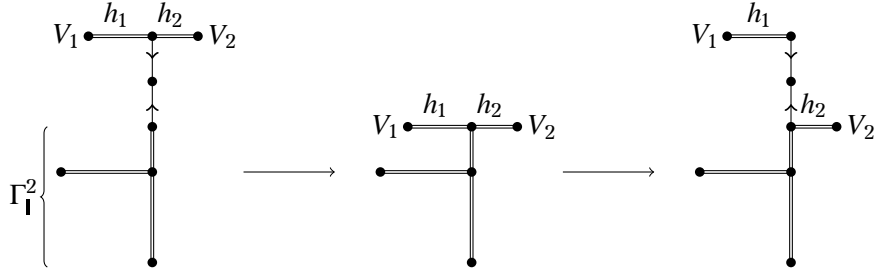
\begin{figure}[ht]
 \centering
\begin{tikzpicture}[scale=1.2,decoration={
    markings,
    mark=at position 0.5 with {\arrow[thick]{>}}}]

    \begin{scope}[xshift=0cm]
\coordinate (p0) at (1,1);
\coordinate (p1) at (0,1);
\coordinate (p2) at (1,0);
\coordinate (p3) at (1,1.5);
\coordinate (p4) at (.3,2.5);
\coordinate (p5) at (1.5,2.5);
\coordinate (q2) at (1,2);
\coordinate (q3) at (1,2.5);

\draw[double] (p0) --(p1)coordinate[pos=.5] (r1);
\draw[double] (p0) --(p2)coordinate[pos=.5] (r2);
\draw[double] (p0) --(p3)coordinate[pos=.5] (r3);
\draw[double] (p4) --(p5)coordinate[pos=.25] (r4)coordinate[pos=.75] (r5);
\draw[postaction={decorate}] (p3) -- (q2) coordinate[pos=.5] (s1);
\draw[postaction={decorate}] (q3) -- (q2) coordinate[pos=.5] (s2);

      \foreach \i in {0,1,...,5}
   \fill (p\i)  circle (1.5pt);
     \fill (q2)  circle (1.5pt);
     \fill (q3)  circle (1.5pt);
  
\node[above] at (r4) {$h_{1}$};
\node[above] at (r5) {$h_{2}$};
 \node[left] at (p4) {$V_{1}$};
\node[right] at (p5) {$V_{2}$}; 

   \draw [decorate,decoration={brace}]	(-.15,-.05) -- (-.15,1.5) node [midway,left] {$\Gamma_\vert^2$};

          \draw[->] (2,1) -- (3,1);     
    \end{scope}

  \begin{scope}[xshift=3.5cm]
\coordinate (p0) at (1,1);
\coordinate (p1) at (0,1);
\coordinate (p2) at (1,0);
\coordinate (p3) at (1,1.5);
\coordinate (p4) at (.3,1.5);
\coordinate (p5) at (1.5,1.5);

\draw[double] (p0) --(p1)coordinate[pos=.5] (r1);
\draw[double] (p0) --(p2)coordinate[pos=.5] (r2);
\draw[double] (p0) --(p3)coordinate[pos=.5] (r3);
\draw[double] (p4) --(p5)coordinate[pos=.25] (r4)coordinate[pos=.75] (r5);

      \foreach \i in {0,1,...,5}
   \fill (p\i)  circle (1.5pt);

\node[above] at (r4) {$h_{1}$};
\node[above] at (r5) {$h_{2}$};
 \node[left] at (p4) {$V_{1}$};
\node[right] at (p5) {$V_{2}$}; 
      
\draw[->] (2,1) -- (3,1);     
    \end{scope}

      \begin{scope}[xshift=7cm]
\coordinate (p0) at (1,1);
\coordinate (p1) at (0,1);
\coordinate (p2) at (1,0);
\coordinate (p3) at (1,1.5);
\coordinate (p4) at (.3,2.5);
\coordinate (p5) at (1.5,1.5);
\coordinate (q2) at (1,2);
\coordinate (q3) at (1,2.5);

\draw[double] (p0) --(p1)coordinate[pos=.5] (r1);
\draw[double] (p0) --(p2)coordinate[pos=.5] (r2);
\draw[double] (p0) --(p3)coordinate[pos=.5] (r3);
\draw[double] (p3) --(p5)coordinate[pos=.5] (r5);
\draw[double] (q3) --(p4)coordinate[pos=.5] (r4);
\draw[postaction={decorate}] (p3) -- (q2) coordinate[pos=.5] (s1);
\draw[postaction={decorate}] (q3) -- (q2) coordinate[pos=.5] (s2);

      \foreach \i in {0,1,...,5}
   \fill (p\i)  circle (1.5pt);
     \fill (q2)  circle (1.5pt);
     \fill (q3)  circle (1.5pt);
  
\node[above] at (r4) {$h_{1}$};
\node[above] at (r5) {$h_{2}$};
 \node[left] at (p4) {$V_{1}$};
\node[right] at (p5) {$V_{2}$}; 
    \end{scope}
\end{tikzpicture}
 \caption{Pumping mass from a pendent vertical segment $[V_1,V_2]=\Gamma^1_\vert $ to the core graph~$\Gamma_\vert^2$.} \label{Pump}
\end{figure}

\subsection{Proof of Lemma~\ref{lmm:Line2Bush}}\label{sec:prooflema}
 Starting with the two bush graph,  we detach $(g-s)$ pairs of consecutive little $\hbar=1/2$ twigs from their root. The graph after detaching the first pair is  shown in the left of 
 Figure~\ref{2Bush}. We do the same for the pairs of consecutive  big $\hbar=1$ twigs and obtain  the graph in the right of Figure~\ref{2Bush}. Finally it suffices to 'rotate' each horizontal segment counterclockwise to obtain the linear graph. The intermediate positions of  
 horizontal components are indicated by dotted/dashed lines on the same Figure~\ref{2Bush}
 on the right. 

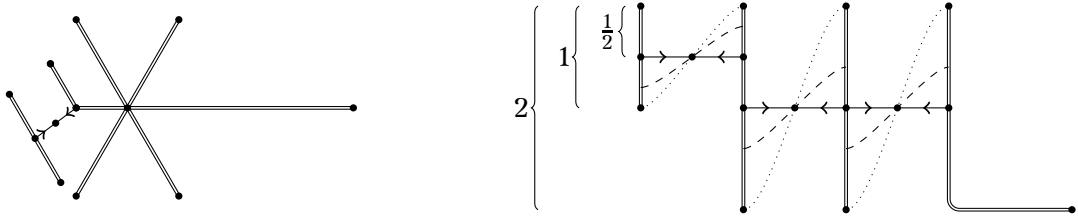
\begin{figure}[ht]
 \centering
\begin{tikzpicture}[scale=1.35,decoration={
    markings,
    mark=at position 0.5 with {\arrow[thick]{>}}}]
   
   \begin{scope}[]
 \coordinate (p0) at (0,0);
 \foreach \i in {1,2,...,5}
\coordinate (p\i) at (60*\i:1);
\coordinate (p3) at (-1/2,0);
\coordinate (q0) at (2.2,0);
          \foreach \i in {1,2,...,5}
   \draw[double] (p0) -- (p\i);
   \coordinate (p6) at (-1/2,0);
   \draw[double] (p6) -- ++(120:.5) coordinate (q1);
   \coordinate (q4) at (-.70,-.15);\coordinate (q5) at (-.90,-.3);
   \draw[postaction={decorate}] (p6) --(q4);   \draw[postaction={decorate}] (q5) --(q4);
   \draw[double] (q5) -- ++(120:.5)coordinate (q2);   \draw[double] (q5) -- ++(-60:.5)coordinate (q3);
   
    \draw[double] (p0) -- (q0);
    \foreach \i in {0,1,...,5}
   \fill (q\i)  circle (1pt);
       \foreach \i in {0,1,...,6}
   \fill (p\i)  circle (1pt);
   \end{scope}

   \begin{scope}[xshift=4cm]
 \foreach \i in {1,2,...,4}
\coordinate (p\i) at (\i,1);
          \foreach \i in {2,3,4}
\coordinate (q\i) at (\i/2,1/2);
          \foreach \i in {1,2,...,8}
\coordinate (r\i) at (\i/2,0);
\coordinate (r7) at (3.5,0);
          \foreach \i in {2,3,4}
\coordinate (s\i) at (\i,-1);
\coordinate (s5) at (5.2,-1);

   \draw[double] (p1) -- (r2);
                \foreach \i in {2,3}
   \draw[double] (p\i) -- (s\i);
       \draw[double,rounded corners] (p4) -- (s4) -- (s5);
      \draw[postaction={decorate}] (q2) -- (q3);\draw[postaction={decorate}] (q4) -- (q3);
      \draw[postaction={decorate}] (r4) -- (r5);\draw[postaction={decorate}] (r6) -- (r5);
\draw[postaction={decorate}] (r6) -- (r7);\draw[postaction={decorate}] (r8) -- (r7);

 \foreach \i in {1,2,...,4}
\fill (p\i) circle (1pt);
 \foreach \i in {1,2,...,4}
\fill (q\i) circle (1pt);
 \foreach \i in {4,5,...,8}
\fill (r\i) circle (1pt);
\fill (r2) circle (1pt);
   \foreach \i in {2,3,5}
 \fill (s\i) circle (1pt);
    \draw [decorate,decoration={brace}]	(0,-1) -- (0,1) node [midway,left] {$2$};
        \draw [decorate,decoration={brace}]	(.4,0) -- (.4,1) node [midway,left] {$1$};
\draw [decorate,decoration={brace}]	(.85,.5) -- (.85,1) node [midway,left] {$\frac{1}{2}$};

\draw[dashed] (1,.2) .. controls ++(0:.2) and ++(180:.2) .. (2,.8);
\draw[dashed] (2,-.4) .. controls ++(0:.2) and ++(180:.2) .. (3,.4);
\draw[dashed] (3,-.4) .. controls ++(0:.2) and ++(180:.2) .. (4,.4);

\draw[dotted] (1,0) .. controls ++(0:.2) and ++(180:.2) .. (2,1);
\draw[dotted] (2,-1) .. controls ++(0:.3) and ++(180:.3) .. (3,1);
\draw[dotted] (3,-1) .. controls ++(0:.3) and ++(180:.3) .. (4,1);
   \end{scope}
   \end{tikzpicture}
 \caption{The two Bush form $\Gamma^*(s,g,n)$ with $s=2$ and $g=3$ just after detaching the first pair of small twigs and its deformation into the standard line form $\Gamma(s,g,n)$. The numbers designate the heights of the edges.} \label{2Bush}
\end{figure}

\smallskip
\par
For the deformation from the two bush graph $\Gamma^*(s,g,n)$ to $\Gamma^*(s-1,g,n)$ we 
detach a bunch of $(g-s)$ pairs of neighbouring twigs from the small bush, roll  
the bunch and attach it to the midpoint of neighbouring twig of the large bush provided $s>0$ as shown on the left of Figure~\ref{2Bush2}. We obtain $s-1$ twigs of unit height 
to the right of the new small bush and $s+1$ twigs of $\hbar=1$ to the right counted from the root of the large bush. Now we detach a couple of neighbouring unit height twigs
from the larger part of the large bush, roll the  $\hbar=2$ pendent vertical segment and attach it to the endpoint of the tail. Since the height of the tail is even we get the 
graph in the right of Figure~\ref{2Bush2}. A unit height edge incident to the endpoint of the tail may be detached, rolled toward the small bush and attached to its root.
Thus we obtain the standard graph $\Gamma^*(s-1,g,n)$.
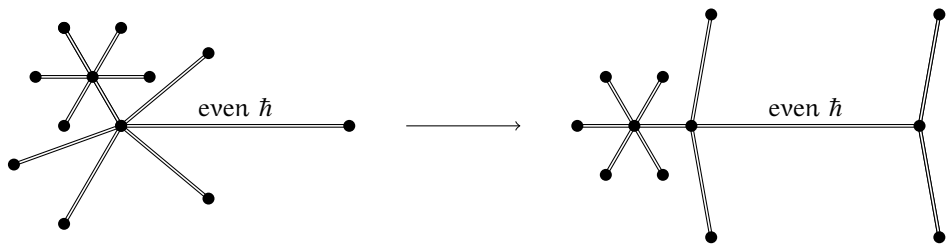
\begin{figure}[ht]
 \centering
\begin{tikzpicture}[scale=1.5,decoration={
    markings,
    mark=at position 0.5 with {\arrow[thick]{>}}}]
 \coordinate (p0) at (0,0);
 \coordinate (p1) at (40:1);
 \coordinate (p2) at (120:1);
 \coordinate (p3) at (200:1);
 \coordinate (p4) at (-120:1);
 \coordinate (p5) at (-40:1);
\coordinate (q0) at (2,0);
          \foreach \i in {1,2,...,5}
   \draw[double] (p0) -- (p\i);
   \coordinate (p6) at (120:1/2);
             \foreach \i in {1,2,3,4,5,6}
   \draw[double] (p6) -- ++(60*\i:.5) coordinate (q\i);
    \draw[double] (p0) --node[above] {\text{\small even $\hbar$}} (q0);
    \foreach \i in {0,1,2,3,4,5,6}
   \fill (q\i)  circle (1.5pt);
       \foreach \i in {0,1,2,3,4,5,6}
   \fill (p\i)  circle (1.5pt);

   \draw[->] (2.5,0) -- (3.5,0);
   
   \begin{scope}[xshift=5cm]
\coordinate (p0) at (0,0);
\coordinate (p2) at (80:1);
\coordinate (p3) at (-80:1);
\coordinate (q0) at (2,0);
 \draw[double] (q0) --++ (80:1) coordinate (p1);
  \draw[double] (q0) --++ (-80:1) coordinate (p4);
          \foreach \i in {2,3}
   \draw[double] (p0) -- (p\i);
             \foreach \i in {1,4}
   \draw[double] (q0) -- (p\i);
   \coordinate (p6) at (-1/2,0);
      \draw[double] (p0) -- (p6);
             \foreach \i in {1,2,3,4,5}
   \draw[double] (p6) -- ++(60*\i:.5) coordinate (q\i);
    \draw[double] (p0) --node[above] {\text{\small even $\hbar$}} (q0);
    \foreach \i in {0,1,2,3,4,5}
   \fill (q\i)  circle (1.5pt);
       \foreach \i in {0,1,2,3,4,6}
   \fill (p\i)  circle (1.5pt);
   \end{scope}

   \end{tikzpicture}
 \caption{An isoperiodic transformation between the two bush graphs $\Gamma^*(s,g,n)$  and $\Gamma^*(s-1,g,n)$ when $s+g+n$ is odd.} \label{2Bush2}
\end{figure}

\smallskip
\par
\subsection{Proof of Theorem~\ref{thm:isoperdef}}
Let  $\Gamma(M)$  be a weighted graph 
 associated to a Riemann surface $M$ of genus  $g>0$
corresponding to Equation~\eqref{PA} with a primitive solution of degree 
$n\geq g+1$. We prove that  $\Gamma(M)$ may be isoperiodically deformed to the two bush graph $\Gamma^*(s,g,n)$ for some $s=0,1,\dots, m^*$, where $m^*:=\min(g-1,n-g-1)$. 

The proof splits into several consecutive steps:
\begin{enumerate}[(1)]
 \item Collapsing of the horizontal component of the graph to obtain a purely vertical graph.
 \item Detaching the vertical segment of minimal possible length $\hbar=1$.
 \item Bringing the core graph to the standard form by recursion on the genus~$g$.
\end{enumerate}

\smallskip
\par
\paragraph{Stage 1: obtaining a purely vertical graph.}
Let $\Gamma(M)$ be any graph satisfying the hypothesis of Theorem~\ref{thm:isoperdef}.
The elimination of its horizontal component may be achieved by linearly decreasing to zero the values of the width function $W(V)$ at all vertices~$V$ of the horizontal subgraph. The only drawback of this deformation is that  some branchpoints may collide in the final instant of the deformation. To prevent it we may preliminary "rotate" every horizontal component of the graph by shifting all its points of intersection with the vertical subgraph by the same small value $\delta h$ and in the same direction to avoid passing through the branchpoints. An example of the rotation is shown in Figure~\ref{fig:rothor}.
After the contraction of its horizontal component, the graph is composed of vertical edges only.

\begin{figure}[ht]
 \centering
\begin{tikzpicture}[scale=1.4,decoration={
    markings,
    mark=at position 0.5 with {\arrow[thick]{>}}}]

\coordinate (p0) at (0,-.6);
\coordinate (p1) at (0,.5);
\coordinate (p2) at (2,-.5);
\coordinate (p3) at (2,.7);
\coordinate (q1) at (1,0);
\coordinate (q0) at (0,0);
\coordinate (q2) at (2,0);
\coordinate (p4) at (1,1);
\coordinate (p5) at (1.4,1);
\coordinate (p6) at (.4,1);
\coordinate (p7) at (1,1.5);

\draw[postaction={decorate}] (q0) -- (q1);\draw[postaction={decorate}] (q2) -- (q1);
\draw[double] (q0) --(p1)coordinate[pos=.25] (s1)coordinate[pos=.7] (s2);\draw[double] (p2) --(p3)coordinate[pos=.3] (s3)coordinate[pos=.2] (s3)coordinate[pos=.8] (s4);
\draw[double] (p5) --(p6);\draw[double] (p4) --(p7);
\draw[postaction={decorate}] (p4) -- (q1);

   \fill (p4)  circle (1.5pt);
         \foreach \i in {0,1,2}
  \fill (q\i)  circle (1.5pt); 
  
  \draw[postaction={decorate},dashed] (0,.2) .. controls ++(0:.2) and ++(120:.2) .. (q1);
    \draw[postaction={decorate},dashed] (2,-.2) .. controls ++(180:.2) and ++(-60:.2) .. (q1);
      \draw[postaction={decorate},dashed] (1.2,1) .. controls ++(-90:.2) and ++(30:.2) .. (q1);
  
  \draw [decorate,decoration={brace}]	(-0.14,0) -- (-0.14,.2) node [midway, left] {$h$};
  \draw [decorate,decoration={brace}]	(2.14,0) -- (2.14,-.2) node [midway, right] {$h$};
  \draw [decorate,decoration={brace}]	(1.02,1.14) -- (1.2,1.14) node [midway, above] {$h$};

\end{tikzpicture}
\caption{Clockwise rotation of a component of $\Gamma_\hor$, new position of the horizontal component is dashed.}
\label{fig:rothor}
\end{figure}
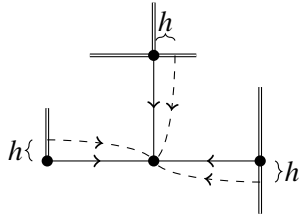

\smallskip
\par
\paragraph{Stage 2: creating a pendent segment of height $\hbar=1$.}

Let us first show that there exist two hanging edges neighbouring with respect to the cyclic order around some vertex~$V$ of the graph. Indeed, take any vertex $L$ of the graph $\Gamma$.  Choose any vertex $V_1$ of~$\Gamma$ at the maximal path length from $L$ (i.e. the number of edges in the path joining them). This is necessarily a hanging vertex of $\Gamma$ and the previous vertex $V$ in the path $[L,V_1]$ on the graph is at distance one less from~$L$. The degree 
$d(V)=d_\vert(V)>1$ since $g>0$, moreover $d(V)\neq2$ due to Property~(T3) of admissible graphs, hence the node $V$ is joined to yet another vertex $V_2$ at the same distance  from $L$ as $V_1$.  This vertex $V_2$ is hanging too, so the edges joining $V$ to $V_1$ and $V_2$ are the desired ones.

 We are going to create a pendent vertical segment of minimal height $\hbar=1$ using the two transformations of \emph{rolling} and \emph{pumping}. Detach  from the rest of the graph the  vertical segment $[V_1,V_2]:=\Gamma^1_\vert$ obtained above. Then
 roll it around the core graph $\Gamma^2_\vert$ and pump its mass whenever possible. Since the height of the  segment is always an integer it cannot diminish ad infinitum. Hence it stabilizes at some integer $l\geq1$.  If $l>1$, then all $\hbar$-distances between neighboring branchpoints on the boundary of the core graph are divisible by $l$ and hence all the periods of $d\eta_M$ lie in the coarse lattice corresponding to the integer $n/l$. This would mean that the solution of degree $n$ of Equation~\eqref{PA} is not primitive.

This  pendent vertical segment of height $\hbar=1$ is called the \emph{catalyst}.
We can roll it to any convenient place of the rest of the graph where is does not interfere with further manipulations. In particular, for $g=1$ we can attach it to the core graph to obtain the two bush graph $\Gamma^{\ast}(0,1,n)$.

\smallskip
\par
\paragraph{Stage 3: Induction Step for $g\geq2$.}
Let us consider the core graph $\Gamma^2_\vert$ obtained after detaching the catalyst 
as a separate graph equipped with the present $\hbar$ heights of vertical edges.
The graph $\Gamma^2_\vert$  corresponds to a Pell-Abel equation admitting degree $n-1$
solution. Once the solution is primitive, we bring the graph to the two bush form $\Gamma^*(s,g-1,n-1)$  by the induction hypothesis. Then we roll the catalyst toward the smaller bush and attach it at its root. Thus we obtain $\Gamma^*(s,g,n)$ with parameter $s$ in the admissible range of values.

Suppose now that the Pell-Abel equation  corresponding to the core graph $\Gamma^2_\vert$
admits a primitive solution of smaller degree  $n'=(n-1)/l$ for some integer $l\ge2$. 
By the induction hypothesis, $\Gamma^2_\vert$ may be isoperiodically transformed into 
a two-bush graph $\Gamma^*(s,g-1,n')$ which however uses another scale for the weights of the edges. To return to our initial units we should multiply all the heights of this graph by the integer factor $l=(n-1)/n'$.

Recall that the catalyst of unit height is joined to the rescaled two bush core graph by a cord. We attach the catalyst to a hanging twig of the smaller bush (they exist in the worst case $s=g-1$) at the distance $\hbar=1 \ge l/2$ from the endpoint.
Then we detach the $\hbar=2$ vertical segment (composed of the catalyst and part of the small bush twig) from the graph. The procedure is allowed even in the worst case $l=2$. In that case the catalyst is attached to the root of the small bush, which is not a branch point.
\smallskip
\par
We claim that the remained core graph ${\Gamma_\vert^2}'$ corresponds to Equation~\eqref{PA} admitting primitive solution of degree $n-2$. Indeed, the $\hbar$ distances between the branchpoints along the boundary of ${\Gamma_\vert^2}'$ are all integer and  include the coprime numbers $l$ and $l-1$ (on the small bush). Now the induction step may be applied again and we replace the core graph~${\Gamma_\vert^2}'$ by the two bush form  $\Gamma^*(s,g-1,n-2)$. The pendent segment of height $\hbar=2$ may be attached by its midpoint to the obtained two-bush form either
\begin{enumerate}[(i)]
 \item at the root of the large bush, as pictured left of Figure~\ref{IndStep1}, or
 \item  at the tip of the large bush twig (this happens only when $s>0$) as shown in the upper left picture of Figure~\ref{IndStep2}.
\end{enumerate}
\smallskip
\par
\paragraph{Case (i):}
We detach a bunch of $(g-s-1)$ pairs of twigs from the small bush, roll them and attach to midpoint of the nearest twig of the large bush, as illustrated in the right of Figure~\ref{IndStep1} where the dashed curves show the final position of the horizontal component.
Thus we obtain the graph $\Gamma^*(s+1,g,n)$. Note that $s+1$ is an admissible value of parameter for given $g$ and $n$ when $s\le \min(g-2, n-g-2)$.

\begin{figure}[ht]
 \centering
\begin{tikzpicture}[scale=1.5,decoration={
    markings,
    mark=at position 0.5 with {\arrow[thick]{>}}}]
 \coordinate (p0) at (0,0);
  \coordinate (p1) at (60:1);
   \coordinate (p2) at (90:1);
\coordinate (p3) at (-80:1);
\coordinate (p4) at (-40:1);
\coordinate (q0) at (2.2,0);
          \foreach \i in {1,2,...,4}
   \draw[double] (p0) -- (p\i);
   \coordinate (p6) at (-1/2,0);
      \draw[double] (p0) -- (p6);
             \foreach \i in {1,2,3,4,5}
   \draw[double] (p6) -- ++(60*\i:.5) coordinate (q\i);
    \draw[double] (p0) -- (q0);
    \foreach \i in {0,1,2,3,4,5}
   \fill (q\i)  circle (1.5pt);
       \foreach \i in {0,1,2,3,4,6}
   \fill (p\i)  circle (1.5pt);
   
   \coordinate (b) at (45:1);
   \coordinate (c) at (45:2);
   \draw[double] (c) -- ++(-45:1) coordinate (a1);
   \draw[double] (c) -- ++(135:1) coordinate (a2);
   \draw[postaction={decorate}] (p0) -- (b);   \draw[postaction={decorate}] (c) -- (b);
      \fill (a1)  circle (1.5pt);      \fill (a2)  circle (1.5pt);
            \fill (c)  circle (1.5pt);      \fill (b)  circle (1.5pt);
            
          \draw[->] (2.7,0) -- (3.7,0);
          
                    \begin{scope}[xshift = 5.5cm]
  \coordinate (p0) at (0,0);
  \coordinate (p1) at (60:1);
   \coordinate (p2) at (90:1);
\coordinate (p3) at (-80:1);
\coordinate (p4) at (-40:1);
\coordinate (q0) at (2.2,0);
          \foreach \i in {1,2,...,4}
   \draw[double] (p0) -- (p\i);
   \coordinate (p6) at (-1/2,0);
      \draw[double] (p0) -- (p6);
             \foreach \i in {5}
   \draw[double] (p6) -- ++(60*\i:.5) coordinate (q\i);
    \draw[double] (p0) -- (q0);

   \coordinate (r) at (-.75,.25);
   \coordinate (s) at (-1,.5);
     \foreach \i in {1,2,3,4}
      \draw[double] (s) --++ (60*\i:1/2)coordinate  (q\i);
   \draw[postaction={decorate}] (p6) -- (r);   \draw[postaction={decorate}] (s) -- (r);
              \fill (r)  circle (1.5pt);      \fill (s)  circle (1.5pt);

     \coordinate (b) at (20:1);
   \coordinate (c) at (40:1);
   \draw[double] (c) --  (p0);
   \draw[double] (b) --  (p0);
            \fill (c)  circle (1.5pt);      \fill (b)  circle (1.5pt);

    \foreach \i in {0,1,2,3,4,5}
   \fill (q\i)  circle (1.5pt);
       \foreach \i in {0,1,2,3,4,6}
   \fill (p\i)  circle (1.5pt);

   \coordinate (t) at (0,.5);    \fill (t)  circle (1.5pt);
\draw[dashed] (t) .. controls ++(180:.3) and ++(60:.3) .. (r) .. controls ++(-120:1) and ++(-150:1) .. (s);
   \end{scope}
   \end{tikzpicture}
 \caption{The induction step in case (i) for $g=5$ and $s=2$.} \label{IndStep1}
\end{figure}
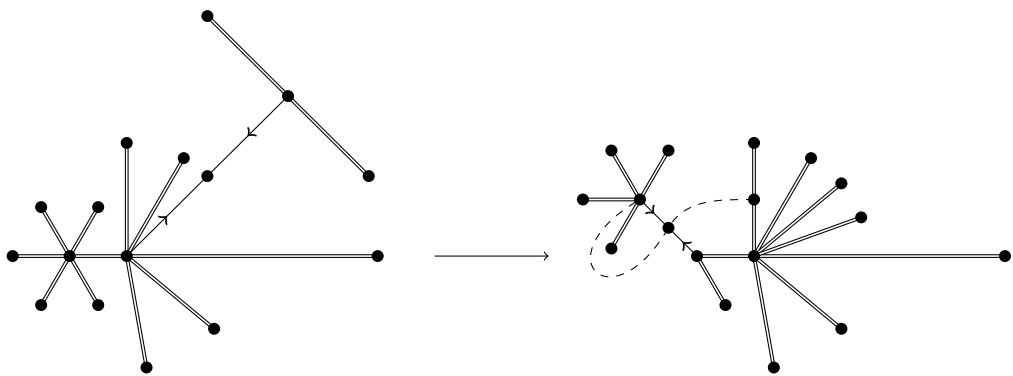
\smallskip
\par
\paragraph{Case (ii):}  We detach one of the unit height twigs at the tip of the large bush twig, roll it around the core graph and attach to the root of the small bush as shown in  the right top picture of Figure~\ref{IndStep2}. Next, we detach the union of the tail (which may be of height~$0$) of two-bush graph and the neighbouring $\hbar=2$ edge. We roll it along the neighboring twig of height $\hbar=1$  as pictured bottom left of Figure~\ref{IndStep2}. Finally we attach it to the core graph at the root of the large bush as pictured bottom right of Figure~\ref{IndStep2}.  The root of the larger bush is now a branch point, so one twig of this bush may be detached and replanted to the smaller bush as we just did. We obtain the two bush graph $\Gamma^*(s-1,g,n)$. This completes the proof.

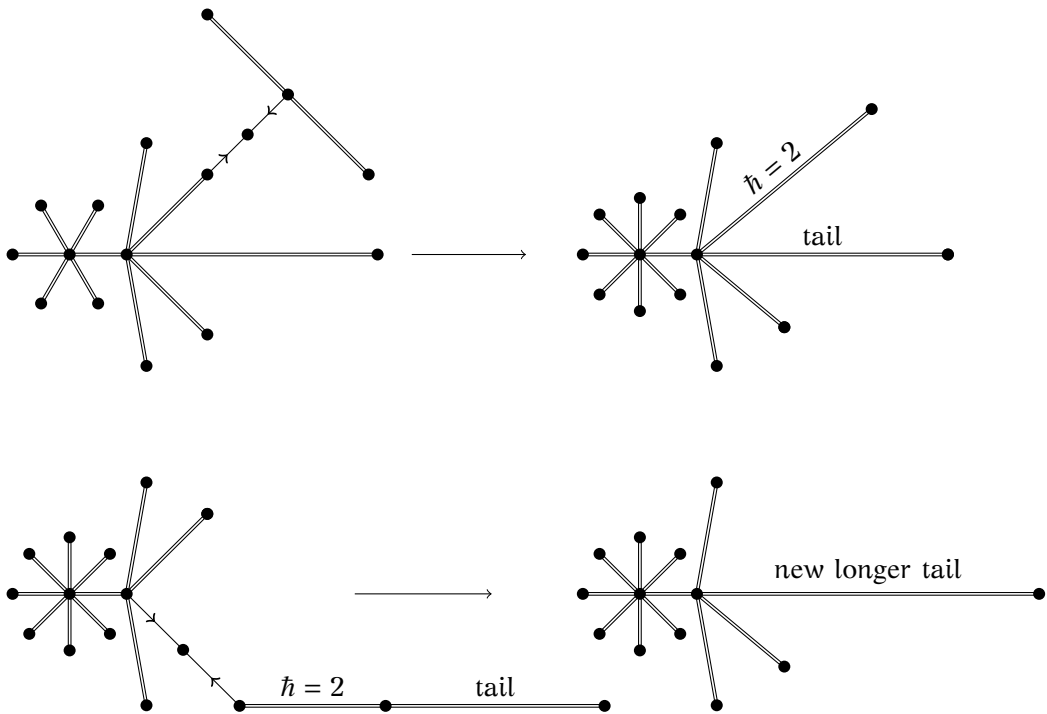
\begin{figure}[ht]
 \centering
\begin{tikzpicture}[scale=1.5,decoration={
    markings,
    mark=at position 0.5 with {\arrow[thick]{>}}}]
 \coordinate (p0) at (0,0);
  \coordinate (p1) at (45:1);
   \coordinate (p2) at (80:1);
\coordinate (p3) at (-80:1);
\coordinate (p4) at (-45:1);
\coordinate (q0) at (2.2,0);
          \foreach \i in {1,2,...,4}
   \draw[double] (p0) -- (p\i);
   \coordinate (p6) at (-1/2,0);
      \draw[double] (p0) -- (p6);
             \foreach \i in {1,2,3,4,5}
   \draw[double] (p6) -- ++(60*\i:.5) coordinate (q\i);
    \draw[double] (p0) -- (q0);
    \foreach \i in {0,1,2,3,4,5}
   \fill (q\i)  circle (1.5pt);
       \foreach \i in {0,1,2,3,4,6}
   \fill (p\i)  circle (1.5pt);
   
   \coordinate (b) at (45:1.5);
   \coordinate (c) at (45:2);
   \draw[double] (c) -- ++(-45:1) coordinate (a1);
   \draw[double] (c) -- ++(135:1) coordinate (a2);
   \draw[postaction={decorate}] (p1) -- (b);   \draw[postaction={decorate}] (c) -- (b);
      \fill (a1)  circle (1.5pt);      \fill (a2)  circle (1.5pt);
            \fill (c)  circle (1.5pt);      \fill (b)  circle (1.5pt);
            
          \draw[->] (2.5,0) -- (3.5,0);
          
                    \begin{scope}[xshift = 5cm]
 \coordinate (p0) at (0,0);
\coordinate (p1) at (40:2);
\coordinate (p2) at (80:1);
\coordinate (p3) at (-80:1);
\coordinate (p4) at (-40:1);
\coordinate (q0) at (2.2,0);
          \foreach \i in {2,...,4}
   \draw[double] (p0) -- (p\i);
      \draw[double] (p0) --node[above,sloped] {$\hbar = 2$} (p1);
   \coordinate (p6) at (-1/2,0);
      \draw[double] (p0) -- (p6);
             \foreach \i in {1,2,...,7}
   \draw[double] (p6) -- ++(45*\i:.5) coordinate (q\i);
    \draw[double] (p0) --node [midway, above] {tail} (q0);
    \foreach \i in {0,1,2,...,7}
   \fill (q\i)  circle (1.5pt);
       \foreach \i in {0,1,2,3,4,6}
   \fill (p\i)  circle (1.5pt);
   \end{scope}
   
   \begin{scope}[yshift=-3cm]
     \coordinate (p0) at (0,0);
  \coordinate (p1) at (45:1);
   \coordinate (p2) at (80:1);
\coordinate (p3) at (-80:1);
          \foreach \i in {1,2,...,3}
   \draw[double] (p0) -- (p\i);
   \coordinate (p6) at (-1/2,0);
      \draw[double] (p0) -- (p6);
               \foreach \i in {1,2,...,7}
   \draw[double] (p6) -- ++(45*\i:.5) coordinate (q\i);
    \foreach \i in {0,1,2,...,7}
   \fill (q\i)  circle (1.5pt);
       \foreach \i in {0,1,2,3,4,6}
   \fill (p\i)  circle (1.5pt);
   
   \coordinate (b) at (-45:.7);
   \coordinate (c) at (-45:1.4);
   \draw[double] (c) -- ++(0:3.2) coordinate (a1)coordinate[pos=.4] (d1)coordinate[pos=.2] (d2)coordinate[pos=.7] (d3);
   \draw[postaction={decorate}] (p0) -- (b);   \draw[postaction={decorate}] (c) -- (b);
      \fill (a1)  circle (1.5pt);          \fill (d1)  circle (1.5pt);      
            \fill (c)  circle (1.5pt);      \fill (b)  circle (1.5pt);
            \node[above] at (d2) {$\hbar=2$};  \node[above] at (d3) {tail};
            
          \draw[->] (2,0) -- (3.2,0);
          
    \begin{scope}[xshift = 5cm]
 \coordinate (p0) at (0,0);
\coordinate (p2) at (80:1);
\coordinate (p3) at (-80:1);
\coordinate (p4) at (-40:1);
\coordinate (q0) at (3,0);
          \foreach \i in {2,...,4}
   \draw[double] (p0) -- (p\i);
   \coordinate (p6) at (-1/2,0);
      \draw[double] (p0) -- (p6);
             \foreach \i in {1,2,...,7}
   \draw[double] (p6) -- ++(45*\i:.5) coordinate (q\i);
    \draw[double] (p0) --node [midway, above] {new longer tail} (q0);
    \foreach \i in {0,1,2,...,7}
   \fill (q\i)  circle (1.5pt);
       \foreach \i in {0,1,2,3,4,6}
   \fill (p\i)  circle (1.5pt);
   \end{scope}
   \end{scope}

   \end{tikzpicture}
 \caption{The induction step in case (ii) for $g=5$ and $s=2$.} \label{IndStep2}
\end{figure}

\begin{rmk}
A proof without recursion is available too, but it is a bit longer and the deformations are more involved. 
\end{rmk}

 \section{Isoperiodic invariants}
\label{sec:GlobInv}
In this section, we show that the genus
$g$ hyperelliptic Riemann surfaces associated to Pell-Abel equations
admitting a primitive solution of degree $n$ corresponding to inequivalent linear graphs
$\Gamma(s,g,n)$ described in Section~\ref{sec:2stdforms} do live in different components of~$\mathscr{A}_{g}^{n}$.
To this end we introduce two global invariants of the isoperiodic transformation
and compute it for all graphs $\Gamma(s,g,n)$.

The first invariant is
based on the partition of the degree
of the polynomial $D(x)$ of Pell-Abel equation~\eqref{PA} into two summands. For this reason we call it the {\em degree partition invariant}. Its elementary construction is
 given in Section~\ref{sec:partition}. We describe a way to compute it using graphs and show that each admissible
partition is realized by a unique linear graph $\Gamma(s,g,n)$.
This completes the proof of Theorem~\ref{main} and the reader could stop there.

The other invariant described in Section~\ref{sec:braid} possesses a much more rich geometric content:
it is related to to braids which describe the motions of unordered
branching sets $\sf E$ in
the plane without collisions of any individual branchpoints.  Hence this invariant is referred as the {\em braid invariant}. The
construction of this invariant is far
less elementary, nonetheless numerically it  coincides  with the degree partition invariant. However it gives a more deep immersion into the geometry of the problem. No doubt, it will be used
for further research in the topic.

\subsection{Degree partition invariant}
\label{sec:partition}
The value of a solution $P$ of Equation~\eqref{PA} at a zero $e\in\sf E$ of the polynomial $D$ may
be either $+1$ or $-1$. 
Therefore, the set $\sf E$ of zeroes of $D$ is split into two subsets ${\sf E}^\pm$.
Since we cannot globally distinguish between solutions $P$ and
$-P$ we consider the cardinalities of those sets as
an unordered partition of $|{\sf E}|=\deg D:=2g+2$. In particular we may assume that
$$
|{\sf E}^-|\le g+1\le |{\sf E}^+| \,.
$$
The \emph{degree partition invariant} of $D\in \mathscr{A}_{g}^{n}$ is the unordered pair 
$(|{\sf E}^-|,|{\sf E}^+|)$ computed for the primitive solution $\pm P(x)$.

This invariant is easily computable from the graph $\Gamma$ of
the associated curve. The \emph{$\hbar$-distance between any two branchpoints of the curve along
the boundary of the graph should be integer. It may be either even or odd depending of whether
those branchpoints lie in the same group ${\sf E}^\pm$ or in different ones.}
To compute the value of the solution $P$ at any branchpoint $e_s$
we use Formula~\eqref{PQ}. The integral of the distinguished differential between
$e_1$ and $e_s$ has been already computed in section~\ref{sec:perhombas} in terms of weights $h$: it is the sum of all vertical weights of the edges on a path between $e_1$ and $e_s$  along the boundary
of~$\Gamma$. Now by substituting the heights $\hbar$ of the
edges instead of their weights~$h$,  we get the justification of the above rule.

\begin{lmm} \label{lem:restdpi}
The degree partition invariant  $(|{\sf E}^-|,|{\sf E}^+|)$ of $D\in \mathscr{A}_{g}^{n}$ having a primitive solution $P$ of degree $n$ satisfies:
\begin{enumerate}[1)]
 \item $|{\sf E}^\pm|>0$,
 \item $|{\sf E}^\pm|\leq n$,
 \item the parity of $|{\sf E}^\pm|$ is equal to the parity of $n$.
\end{enumerate}
\end{lmm}

\begin{proof}
 1) In the graph description, the fact that $|{\sf E}^-|=0$ happens exactly when all boundary $\hbar$
-distances between the branchpoints are even. Dividing all the heights of the graph by $2$, 
we get a solution $P$ of degree  twice less.

2) A nontrivial polynomial has at most its degree number of roots.

3) The set $\sf E$ of roots of $D$ is the set of $x\in \CC$ where $P$
takes value $\pm1$ with odd multiplicity.
\end{proof}

Finally, we compute the degree partition invariant of the linear graphs.
\begin{ex}
The degree partition invariant of the linear graph $\Gamma(s,g,n)$
has the following smaller element: 
\be
\label{DPinv}
|{\sf E}^-|= g-s+\alpha\,,
\ee
where $\alpha=(s+g+n)\mod 2 \in \lbrace 0,1\rbrace$.
We note that all linear graphs $\Gamma(s,g,n)$
correspond to different  partitions (and therefore belong to different
components of $\mathscr{A}_{g}^{n}$)
outside the cases explicitly described in Lemma~\ref{lmm:Line2Bush}.
\end{ex}

\begin{rmk} One can check by direct calculation that the number of invariants is exactly the number $a(g,n)$ of  components in Theorem~\ref{main}. This completes its proof.
\end{rmk}

To conclude this subsection, we compute the partition invariant of some Riemann surfaces defined over $\mathbb{Q}$.
\begin{ex}
In \cite[p. 30]{PlatoTorSurv} it is shown that the Pell-Abel equation~\eqref{PA} with the polynomial
$$D(x) = x^{6} + 6x^{4}+33x^{2}+24$$
has the  primitive solution 
$$P(x)=\frac{1}{24}x^{9} + \frac{3}{8}x^{7}+\frac{9}{4}x^{5}+6x^{3}+9x \text{ and } Q(x)=\frac{1}{24}x^{6} + \frac{1}{4}x^{4}+x^{2}+1 \,.$$
Now it is easy to check numerically that the vector of values of $P$ at the roots of $D$ contains $3$ times~$+1$ and $3$ times $-1$. Hence $D$ belongs to the component of degree partition invariant $(3,3)$.
\end{ex}

\subsection{Braid invariant}
\label{sec:braid}
We know what does the invariance of periods mean for small deformations of
the branching set $\sf E$, see e.g. the discussion at the end of
Section~\ref{sec:ModuliSpace}.  For large deformations, we should
somehow identify the integration cycles on remote surfaces $M({\sf
E})$. This is done via the parallel transport of cycles by the
Gauss-Manin connection  (see \cite[Section~I.1]{Vas} or
\cite[Chapter~5]{Bbook}).

Suppose that we move the branchpoints and simultaneously distort
a cycle $C$ so that the branchpoints  never cross its projection to
the $x$-plane.  In this way we transport
the cycle along some path $\tau$ in the space  $\tilde{\cal H}_g$ of
hyperelliptic Riemann surfaces with a pair of marked point at infinity
(identified with the space of complex monic square free polynomials
$D(x)$ of degree $2g+2$). The resulting cycle
 belongs to the  Riemann surface corresponding to the end of the path
and we denote it as $C\cdot\tau$,
 whereas $C$ itself belongs to the  Riemann surface at the beginning
of the path. This action of paths on the homology
 spaces of the  Riemann surfaces in $\tilde{\cal H}_{g}$ is
associative: $C\cdot(\tau\cdot\sigma)=(C\cdot\tau)\cdot\sigma$
provided all products are correctly defined, e.g. the end of $\tau$ is
the beginning of $\sigma$, etc.

\subsubsection{Braids and isoperiodic deformations}
Fix  an affine hyperelliptic  Riemann surface $M_1$ whose
branchpoints $e_{1}<e_{2}<\cdots<e_{2g+2}$ are real. We introduce the
standard homology basis
$C_1,C_2,\dots,C_{2g+1}$ of $H_{1}(M_{1},\ZZ)$, where the projection
of $C_{i}$ to the $x$-plane encircles $e_{i}$ and $e_{i+1}$, as
pictured in the left panel of Figure~\ref{RealHom}.
Any  Riemann surface~$M_2$ of the same genus $g$ with purely real
branchpoints may be connected to $M_1$ by a path $\sigma$ in the space
$\tilde{\cal H}_g$ such that all the branchpoints move along the real
axis during the deformation. The transport of the standard homology
basis for the starting surface along~$\sigma$ is the  standard basis
for the ending surface.  Note that~$\sigma$ usually does not conserve
any period.

Suppose an  isoperiodic path $\tau$ in~$\tilde{\cal H}_g$
connects $M_1$ to $M_2$, both with real branchpoints (intermediate
Riemann surfaces of the path may have general branchpoints). Let
$d\eta_j$ be the distinguished differential on~$M_j$ defined in
Section~\ref{sec:solv}. For every cycle $C_j\in H_1(M_1,\ZZ)$ the
equalities hold:
\be
\int_{C_j} d\eta_1=\int_{C_j\cdot\tau} d\eta_2=
\int_{C_j\cdot(\tau\cdot\sigma^{-1})\cdot\sigma} d\eta_2=
\sum_{r=1}^{2g+1}B_{jr}(\tau\cdot\sigma^{-1})\int_{C_r\cdot\sigma} d\eta_2 \,.
\ee
The path $\beta:=\tau\cdot\sigma^{-1}$ here is a loop in the space
$\tilde{\cal H}_g$
with the base point $M_1$ and it is represented by a braid $\beta\in \Br_{2g+2}$
on $2g+2$ strands. The transport of cycles along the loops by the
Gauss-Manin connection has nontrivial holonomy. Given a standard basis
of $H_1(M_1,\ZZ)$, the holonomy is given by the matrix
$B(\beta)=||B_{jr}||\in \SL_{2g+1}(\ZZ)$. It is easy to calculate this
matrix for an elementary braid
$\beta_r$ corresponding to  Dehn half twist \cite{Birman} interchanging
the branchpoints $e_r$ and $e_{r+1}$ counterclockwise for
$r=1,2,\dots,2g+1$ (see right panel of Figure~\ref{RealHom}):
\be
(C_1,C_2,\dots,C_{2g+1})\cdot\beta_r=(\dots,C_{r-1}-C_r,C_r,C_{r+1}+C_r,\dots).
\label{Burep}
\ee
The braid $\beta_r$ changes only two homology cycles $C_{r-1}$ and $C_{r+1}$.
This matrix representation of braids group is known as the {\em reduced Burau
representation $\mathcal{B}_{t}$} (see Section~2 of \cite{GGBraids})
evaluated at
the parameter $t=-1$.
  \begin{figure}[htb]
\center
 \begin{tikzpicture}[scale=1,decoration={
    markings,
    mark=at position 0.35 with {\arrow[very thick]{>}}}]

\begin{scope}
     \fill[fill=black!10] (-1.5,0) coordinate (Q)  ellipse (3.2cm and 1.3cm);
\draw[thick,gray] (-3.9,0) coordinate (a0)-- (-2.1,0) coordinate(b0)
coordinate[pos=.5](c0);
\draw[thick,gray] (-.9,0)  coordinate (a1)-- (.9,0) coordinate (b1)
coordinate[pos=.5](c1);
   \fill (a0)  circle (1.5pt);
   \fill (b0)  circle (1.5pt);\node[below right] at (a1) {$e_{j+1}$};
  \fill (a1)  circle (1.5pt);\node[below right] at (b0) {$e_{j}$};
   \fill (b1)  circle (1.5pt);

 \draw[postaction={decorate},blue] (-3,0) circle [x radius=1.1cm, y
radius=.4cm];
\node[blue] at (-2.6,.6) {$C_{j-1}$};
   \draw[postaction={decorate},blue] (0,0) circle [x radius=1.1cm, y
radius=.4cm];
\node[blue] at (.2,.6) {$C_{j+1}$};

\draw[postaction={decorate},red] (-.53,0) .. controls ++(90:.8) and
++(90:.8) .. (-2.4,0);
\draw[red,dashed] (-2.4,0) .. controls ++(-90:1) and ++(-90:1) .. (-.53,0);
 \node[red] at (-1.7,.8) {$C_{j}$};
 \draw[red,dashed]  (-4.5,-.4) .. controls ++(0:.4) and ++(-90:.3) .. (-3.7,0);
    \draw[postaction={decorate},red] (-3.7,0)  .. controls ++(90:.3)
and ++(0:.4) .. (-4.5,.4);
   \node[red] at (-4.15,.6) {$C_{j-2}$};
      \draw[red,dashed]  (1.5,-.4) .. controls ++(180:.4) and
++(-90:.3) .. (.7,0);
    \draw[postaction={decorate},red] (1.5,0.4)  .. controls++(0:-.4)
and  ++(90:.3) .. (.7,0);
   \node[red] at (1.4,.65) {$C_{j+2}$};

 \draw[->] (1.8,0) --node[above]{$\beta_{j}$} (2.6,0) coordinate[pos=.5](z);
\end{scope}

\begin{scope}[xshift=7.4cm]
     \fill[fill=black!10] (-1.5,0) coordinate (Q)  ellipse (3.2cm and 1.3cm);
\draw[thick,gray] (-3.9,0) coordinate (a0)-- (-2.1,0) coordinate(b0)
coordinate[pos=.5](c0);
\draw[thick,gray] (-.9,0)  coordinate (a1)-- (.9,0) coordinate (b1)
coordinate[pos=.5](c1);
   \fill (a0)  circle (1.5pt);
   \fill (b0)  circle (1.5pt);
  \fill (a1)  circle (1.5pt);
   \fill (b1)  circle (1.5pt);

    \draw[postaction={decorate},blue] (-3.75,0) .. controls ++(120:.2)
and ++(90:.3) .. (-4.2,0) .. controls ++(-90:1) and ++(-60:1) ..
(-0.7,0);
    \draw[dashed,blue] (-0.7,0) .. controls ++(120:.4) and ++(0:.8) ..
(-2.5,-.3) .. controls ++(180:.6) and ++(-60:.2) .. (-3.75,0);
\node[blue] at (-3.9,-.8) {$C_{j-1}'$};

   \draw[postaction={decorate},blue]  (0.75,0) .. controls ++(-60:.3)
and ++(-90:.2) .. (1.1,0) .. controls ++(90:.7) and
++(120:.4)..(-2.25,0);
    \draw[dashed,blue] (.75,0) .. controls ++(120:.5) and ++(10:.1) ..
(-1,.2) .. controls ++(-190:.1) and ++(-60:.4) .. (-2.25,0);
\node[blue] at (.3,.7) {$C_{j+1}'$};

\draw[postaction={decorate},red] (-.5,0) .. controls ++(90:.8) and
++(90:.8) .. (-2.4,0);
\draw[red,dashed] (-2.4,0) .. controls ++(-90:1) and ++(-90:1) .. (-.5,0);
 \node[red] at (-1.7,.8) {$C_{j}$};
  \draw[red,dashed]  (-4.5,-.4) .. controls ++(0:.4) and ++(-90:.3) .. (-3.7,0);
    \draw[postaction={decorate},red] (-3.7,0)  .. controls ++(90:.3)
and ++(0:.4) .. (-4.5,.4);
   \node[red] at (-4.15,.6) {$C_{j-2}$};
      \draw[red,dashed]  (1.5,-.4) .. controls ++(180:.4) and
++(-90:.3) .. (.7,0);
    \draw[postaction={decorate},red] (1.5,0.4)  .. controls++(0:-.4)
and  ++(90:.3) .. (.7,0);
   \node[red] at (1.4,.65) {$C_{j+2}$};
\end{scope}
 \end{tikzpicture}
\caption{The standard homology basis for a purely real  Riemann
surface $M$ on the left and the transport of basic cycles under the
Dehn half-twist on the right. The slits pairwise joining the branch
points are pictured in grey.}
\label{RealHom}
\end{figure}
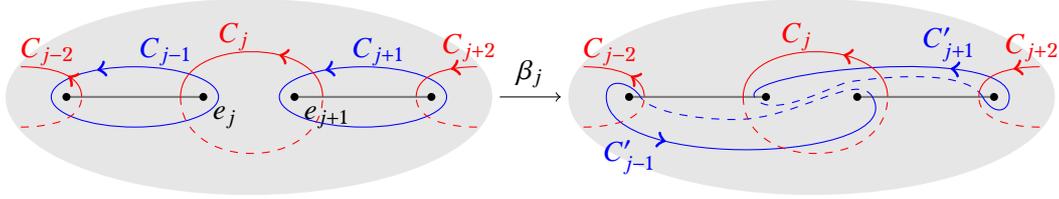

It follows from this discussion that the naturally ordered periods of
two linear graphs connected by an isoperiodic deformation lie in the
same orbit of the representation~$B(\beta)$.
However the braid group is infinite and the fact that two vectors
belong to the same orbit is difficult to check. For this reason we
consider a coarser invariant.  We consider the binary arrays of length
$2g+1$. Obviously, the Burau representation modulo $2$ acts on such
binary strings too, but any orbit is now finite. We will be interested
in the orbits of the binary arrays of the form
\be
(\hbar_1\hbar_2\hbar_3\dots\hbar_{2g+1})\mod2, \text{ with }
\hbar_r:=\frac{n}{2\pi i}\int_{C_r}d\eta_M \in\ZZ
\ee
being the rescaled periods of the distinguished differential,
 $r=1,2,\dots,2g+1$. Note that for totally real curves $M$ all
entries~$\hbar_r$ with even indexes $r$ are zeros and the total sum of
$\hbar_r$ is $n$.

Our immediate goal is to learn how to distinguish
orbits of Burau representation reduced mod 2 on the binary arrays.

\subsubsection{Orbits of the Burau action reduced modulo $2$}
Consider the following generating set in $\ZZ_2^{2g+1}$:
\be
\begin{array}{r c l}
v_1 &=&(1000000\dots0000);\\
v_2&=&(1100000\dots0000);\\
v_3&=&(0110000\dots0000);\\
v_4&=&(0011000\dots0000);\\
&\vdots&\\
v_{2g+1}&=&(000000\dots0011);\\
v_{2g+2}&=&(000000\dots0001).\\
\end{array}
\ee
The only nontrivial linear relation between those vectors is
$\sum_{r=1}^{2g+2}v_r=0$.
An elementary braid $\beta _r$ acting on this set via the reduced
Burau representation modulo $2$
behaves like a transposition of two neighbouring elements:
\begin{equation*}
v_rB(\beta_r)=v_{r+1},\,  v_{r+1}B(\beta_r)=v_r, \text{ and }
v_jB(\beta_r)=v_j, \text{ when } j\neq r, r+1\,.
\end{equation*}
 Therefore the braid group acts as a permutation group on the elements
$v_{i}$ of the generating set. It follows that the length~$Q$ of the
shortest decomposition (there are exactly two of them) of the elements
$v\in\ZZ_2^{2g+1}$
into the generators $v_r$ with $r=1,\dots,2g+2$ is the only invariant of
our braid action on binary strings. This number $Q$ is the \emph{braid invariant} of the array. Note that it takes value in
$\lbrace 1,2,\dots,g+1 \rbrace$ and distinguishes the orbits of action of Burau
representation of braids on binary arrays.

\begin{rmk}
 Looking more carefully at its action on the set of generators
$v_{i}$, it can be shown that the  group generated by the reduced
Burau matrices reduced mod $2$ in $\SL_{2g+1}(\ZZ_2)$ is isomorphic to
the symmetric group on $2g+2$ elements.
\end{rmk}

\subsubsection{The braid invariant of standard forms}
Let us calculate the value of the braid invariant $Q$ for the hyperelliptic curves with associated
linear graphs $\Gamma(s,g,n)$ for $s=0,\dots, m^*$ where
$m^*:=\min(g-1, n-g-1)$ (recall Remark~\ref{rem:mstar} for the
justification of the definition of $m^*$). The binary array
corresponding to the latter graph is $W_{g-s}$, where
\be
W_s=(1010101\dots0101000\dots000b),
 \text{ where } b(s):=(n+s) \mod 2\,,
\ee
with exactly $s$ entries $1$ in the first $2g$ places.
These vectors satisfy the recurrence relation $W_s=v_{2s-1}+v_{2s-2}+W_{s-2}$
which together with the initial conditions   $W_1=v_1+bv_{2g+2}$ and
$W_2=v_2+v_3+bv_{2g+2}$
gives us the value of the braid invariant of the vectors $W_{s}$.
Indeed, let $\alpha := (s+n+g)\mod 2$ with values $0$ and $1$, then
the invariant is
\be
\label{BBinv}
Q(W_{g-s})=g-s+\alpha\le g+1.
\ee
Hence, the values of $Q$ coincide for the equivalent graphs
$\Gamma(s,g,n)$ and
$\Gamma(s-1,g,n)$ when $g+n+s$ is odd and are different for all the
other graphs.
\smallskip
\par
We conclude by comparing the braid invariant with the degree partition invariant.
\begin{prop}
 The braid invariant $Q$ of the vector of $\hbar$-heights of the linear graph coincide with the number $|E^{-}|$ of the degree partition invariant $(|E^{-}|,|E^{+}|)$.
\end{prop}

\begin{proof}
It suffices to compare  Formula~\eqref{BBinv} with Formula~\eqref{DPinv}.
\end{proof}

\section{$k$-differentials on hyperelliptic Riemann surfaces}
\label{sec:compkdiffs}

In this last section, we prove Corollary~\ref{cor:CCkdiff}. We begin by recalling some known facts on $k$-differentials and their moduli spaces. More information can be find in~\cite{BCGGM3}.

Given integers $g\geq0$ and $k\geq1$, a $k$-differential~$\xi$ on a genus $g$ Riemann surface~$M$ is a non-zero section of the $k$th tensorial product of the canonical bundle $K_{M}$. A $k$-differential is said to be primitive if it is not the power of a $k'$-differential with $k'<k$.  

Given a partition $\mu = (m_{1},\dots,m_{n})$ of $k(2g-2)$, we consider the moduli spaces of $k$-differentials whose orders of zeros are equal to $m_{1},\dots,m_{n}$. This moduli space is called a {\em stratum} of $k$-differentials of type $\mu$ and is denoted $\komoduli(\mu)$. The sublocus parametrizing the primitive $k$-differentials of type $\mu$ is denoted by $\komoduli(\mu)^{\rm prim}$. 
\smallskip
\par
We now compute the number of the connected components of the restriction of the strata of $k$-differentials with a unique zero to the hyperelliptic locus. 

\begin{prop}\label{prop:cckdiffggeq3}
 For $g\geq 2$, the number of connected components of the restriction of $\komoduli(k(2g-2))^{\rm prim}$ to the hyperelliptic locus is
 \begin{itemize}
  \item $\left[\tfrac{g-1}{2}\right]$ if $k=2$;
  \item $1$ if $k=3$ and either $g=2$ or $g=3$;
  \item $g/2$  if $k\geq 4$ and $g\geq 2$ are even;
  \item $g/2+1$  if either $g=2$ and $k\geq 5$ is odd, or $k\geq 3$ is odd and $g\geq 4$ is even;
    \item $(g+1)/2$  if  $g\geq 3$ is odd, $k\neq 2$ and either $g$ or $k$ is not equal to $3$.
 \end{itemize}
\end{prop}

\begin{proof}
 A primitive $k$-differential on a hyperelliptic genus $g$  Riemann surface with a unique zero of order $2k(g-1)$ is equivalent to a primitive solution of Equation~\eqref{PA} of degree $n=k(g-1)$.
Indeed, consider a solution of degree $n$ of the Pell-Abel equation. According to point 3) of Remark~\ref{rem:solve}, there exists a  hyperelliptic Riemann surface~$M_{\infty}$ such that
  \begin{equation}\label{eq:zerosn}
   n\infty_{+} - n\infty_{-} \sim  \calO \,,
  \end{equation}
 where $\calO$ is the trivial bundle of $M_{\infty}$. Moreover, by primitivity of the solution this equation is not satisfied  for any $ n' < n$.
Since we know that
\begin{equation}\label{eq:cano}
   (g-1)\infty_+ + (g-1)\infty_- \sim K \,,
  \end{equation}
  where $K$ is the canonical bundle of $M_{\infty}$,
   we obtain
  \begin{equation}\label{eq:kcano}
   2n \cdot \infty_+ = kK \,.
  \end{equation}
Therefore $\infty_+$ is the unique zero of a $k$-differential~$\xi$. The fact that $n$ is minimal for this property implies that~$\xi$ is a primitive $k$-differential in the locus $\komoduli(2k(g-1))^{\rm prim}$.
  
Conversely, consider a primitive $k$-differential $(M,\xi)$ in the hyperelliptic locus of $\komoduli(2k(g-1))^{\rm prim}$. The zero $z$ of $\xi$ satisfies Equation~\eqref{eq:kcano}. Now it suffices to subtract~$k$ times Equation~\eqref{eq:cano} to this equation to obtain Equation~\eqref{eq:zerosn}. Recall that the degree of the solutions associate to the point $z$ forms a semi-group generated by one element.  Together with the primitivity of the $k$-differential this implies the primitivity of the solution associated to Equation~\eqref{eq:zerosn}.
  
Hence, the components are in one-to-one correspondence with components of primitive solutions of Pell-Abel equation of degree $n=k(g-1)$. Note that $g > n-g-1$ if and only if $k < \tfrac{2g+1}{g-1}$. For $g\geq 3$, this happens if and only if $k=2$ or $k=g=3$. 

Since for $g=2$ we obtain a bijection between the components of $\komoduli[2](2k)^{\rm prim}$ and the components of primitive solutions of Pell-Abel of degree $k$, we obtain the result in genus~$2$ directly from Theorem~\ref{main}.

So if $k=2$, we have $\min(g,2(g-1)-g-1)=g$ and using Theorem~\ref{main}, we obtain that the number of connected components is equal to $\left[(g-1)/2\right]$. If $g=k=3$, the restriction of the stratum $\Omega^{3}\moduli[3](12)^{\rm prim}$ to the hyperelliptic locus is connected. If we are not in one of the previous cases, then the number of components is 
 \begin{eqnarray*}
 \left[\frac{g}{2}\right] + 1\,, & \text{ when $kg-k+g$ is odd, and}\\
  \left[\frac{g+ 1}{2}\right] \,, & \text{ when $kg-k+g$ is even.}
 \end{eqnarray*}
The second case occurs when both $k$ and $g$ are even and the first case otherwise. This concludes the proof of Proposition~\ref{prop:cckdiffggeq3}.
  \end{proof}

  Since Riemann surfaces of genus $2$ are hyperelliptic, this implies the first part of Corollary~\ref{cor:CCkdiff}. Moreover this shows that parity invariant of \cite[Theorem~1.2]{chgeCC} classifies the connected components of $\komoduli[2](2k)^{\rm prim}$. Recall that the parity invariant is given by the parity of the spin structure of the canonical cover associated to a $k$-differential (see Section~5 of \cite{chgeCC} for a detailed discussion).  We now relate the parity invariant  with the degree partition invariant, proving the second part of the corollary.
\begin{prop}\label{prop:relationplatcheby}
 Let $k\geq 5$ be an odd number. The component of $\komoduli[2](2k)^{\rm prim}$ with odd, resp. even, parity corresponds to the component of invariant $(1,5)$, resp. $(3,3)$. Moreover, the component of $\komoduli[2](2k)^{\rm prim}$ is odd if and only if there exists a \Weierstrass point  such that its difference with the zero of the $k$-differential is a $k$-torsion.
\end{prop}
The proof relies on the technology of degenerations that were introduced in \cite{BCGGM3} and studied in Sections~2 and~3.2 of \cite{chgeCC}. It is recommended but not necessary to have some familiarity with these notions: we will only use the notion of twisted $k$-differentials which appears has limit of $k$-differentials.
\begin{proof}
Let $k$ be an odd integer $\geq5$.
 Let  $(M,\xi)\in\komoduli[2](2k)$  be the primitive $k$-differential whose unique zero is~$z$ such that the graph associated to $M$ (as explained in Section~\ref{sec:graphcalcul}) is linear.
It is shown in proof of \cite[Theorem 3]{GePA}  that there is a \Weierstrass point $W\in M$ such that the difference $W-z$ is a $k$-torsion if and only if the linear graph has heights $(2,2,k-2)$. The degree partition invariant of this graph is $(1,5)$ (and of course~$W$ is the preimage of the unique $e \in \sf E^{-}$). 
\smallskip
\par
Let   $(M,\xi)\in\komoduli[2](2k)$ be a primitive $k$-differential of  odd parity and denote by $z$ its zero. It suffices to prove that there exists a \Weierstrass point $W$ such that $W-z$ is $k$-torsion.
\par 
We start we a twisted $k$-differential $(M_{0},\xi_{0})$ obtained by gluing the $k$-th power of an holomorphic differential on a genus $1$ Riemann surface $(M_{1},\omega_{1})$ to the pole of a $k$-differential $(M_{2},\xi_{2})$ in $\komoduli[1](2k,-2k)^{\rm prim}$ whose $k$-residue vanishes  (see Lemma~5.9 of~\cite{chgeCC} for the existence of such $k$-differential). We denote by $z$ the zero of $\xi_{2}$. Remark that the Jacobian of the underlying singular curve $M_{0}$ is the product of the elliptic curves. This twisted $k$-differential $(M_{0},\xi_{0})$ and its Jacobian are sketched in  Figure~\ref{fig:jacobienne}. 

 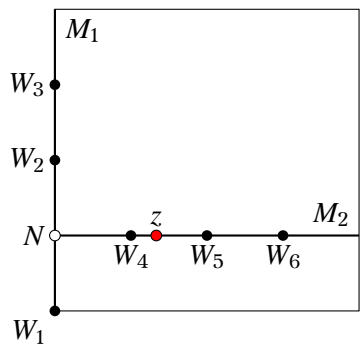
\begin{figure}[htb]
 \centering
\begin{tikzpicture}[scale=1]

\draw (0,0) coordinate (W1) -- ++(4,0) -- ++(0,4) --++(-4,0) coordinate (Q) --++(0,-4);
\draw[thick] (W1) -- (Q)coordinate[pos=.25] (P)coordinate[pos=.5] (W2)coordinate[pos=.75] (W3);
\draw[thick] (P) -- ++(4,0) coordinate[pos=.25] (W4)coordinate[pos=.5] (W5)coordinate[pos=.75] (W6)coordinate (R);

\foreach \i in {1,...,6}
 \fill (W\i) circle  (2pt); 
 
 \filldraw[fill=white] (P) circle (2pt);
 \filldraw[fill=red]  (4/3,1) coordinate (Pk) circle (2pt);
 
 \node[below left] at (W1) {$W_{1}$};
  \node[left] at (W2) {$W_{2}$};
  \node[left] at (W3) {$W_{3}$};
   \node[below] at (W4) {$W_{4}$};
  \node[below] at (W5) {$W_{5}$};
  \node[below] at (W6) {$W_{6}$};
    \node[left] at (P) {$N$};
      \node[above] at (Pk) {$z$};
     \node[below right] at (Q) {$M_{1}$};
   \node[above left] at (R) {$M_{2}$};
\end{tikzpicture}
\caption{The Jacobian of $M_{0}$.}
\label{fig:jacobienne}
\end{figure}

This twisted differential is smoothable in the stratum $\komoduli[2](2k)$. The limits of the \Weierstrass points of any such smoothing are the $2$-torsion points modulo the node $N$. Denote by $W_{1},W_{2},W_{3}$, resp. $W_{4},W_{5},W_{6}$, the $2$-torsion points on $M_{1}$, resp. $M_{2}$. 
We consider the $2$-torsion points on $M_{2}$. Let $v_{1},v_{2}\in\CC$ such that $M_{2}\sim \CC/(\ZZ v_{1}\oplus \ZZ v_{2})$ and suppose that the node is the image of $0\in\CC$.
The coordinates of $z$ are $\left(\tfrac{n_{1}}{2k},\tfrac{n_{2}}{2k}\right)$ where $\pgcd(n_{1},n_{2},2k)\in \lbrace 1,2 \rbrace$ is the rotation number  of $\xi_{2}$ (see \cite[Theorem 3.12]{chgeCC}). Hence, the differences $W_{i}-z$ are given by $\left(\tfrac{n_{1}-k\delta_{1}}{2k},\tfrac{n_{2}-k\delta_{2}}{2k}\right)$ with $(\delta_{1},\delta_{2})\in (\ZZ/2\ZZ)^{2}\setminus\lbrace(0,0)\rbrace$. The orders of torsion of these differences are $$\frac{2k}{\pgcd(n_{1}-k\delta_{1},n_{2}-k\delta_{2},2k) } \,.$$

Suppose that the rotation number $\pgcd(n_{1},n_{2},2k)$ of  $\eta_{2}$ is $1$. If both  $\delta_{i}$ have the same parity than $n_{i}$, then both $n_{i}-k\delta_{i}$ are even. Hence there exists a $2$-torsion point on $M_{2}$, given by $kz \in M_{2}$, such that its difference with $z$ is $k$-torsion. 
Finally, Lemma~5.6 of \cite{chgeCC} shows that the parity of the $k$-differentials obtained by smoothing this twisted $k$-differentials is  odd.
\end{proof}

\phantomsection
 \addcontentsline{toc}{section}{References}
\bibliographystyle{alpha}
\bibliography{biblio}

\newcommand{\etalchar}[1]{$^{#1}$}
\begin{thebibliography}{{McM}06}

\bibitem[{Abe}26]{Abel}
Niels {Abel}.
\newblock {Ueber die Integration der Differential-Formel \(\frac{\rho dx}{\sqrt
  R}\), wenn \(R\) und \(\rho\) ganze Functionen sind}.
\newblock {\em {J. Reine Angew. Math.}}, 1:185--221, 1826.
\newblock French translation, Œuvres complètes, tome 1, p. 104--144, 1881.

\bibitem[BCG{\etalchar{+}}19]{BCGGM3}
Matt Bainbridge, Dawei Chen, Quentin Gendron, Samuel Grushevsky, and Martin
  M{\"o}ller.
\newblock Strata of $k$-differentials.
\newblock {\em {Algebr. Geom.}}, 6(2):196--233, 2019.

\bibitem[BCZ22]{BSZ22}
Fabrizio Barroero, Laura Capuano, and Umberto Zannier.
\newblock Betti maps, {Pell} equations in polynomials and almost-{Belyi} maps.
\newblock {\em Forum Math. Sigma}, 10:23, 2022.
\newblock Id/No e84.

\bibitem[BE01]{BelEno}
Eugene Belokolos and Viktor Enolski{\u{\i}}.
\newblock Reduction of abelian functions and algebraically integrable systems.
  {I}.
\newblock {\em J. Math. Sci., New York}, 106(6):3395--3486, 2001.

\bibitem[BG23]{BoGeShort}
Andrei {Bogatyr\"ev} and Quentin {Gendron}.
\newblock Number of components of {P}ell–{A}bel equations with primitive
  solution of given degree.
\newblock {\em Uspekhi Mat. Nauk}, 78(1):209--210, 2023.
\newblock English translation in Russ. Math. Surv. 78:1 208-210.

\bibitem[{Bir}75]{Birman}
Joan {Birman}.
\newblock {\em Braids, links, and mapping class groups. {Based} on lecture
  notes by {James} {Cannon}}, volume~82 of {\em Ann. Math. Stud.}
\newblock Princeton University Press, 1975.

\bibitem[{Bog}99]{B99}
Andrei {Bogatyr\"ev}.
\newblock {Effective computation of Chebyshev polynomials for several
  intervals}.
\newblock {\em {Sb. Math.}}, 190(11):15--50, 1999.
\newblock English translation, Mat. Sb. 190, No. 11, 1571--1605, 1999.

\bibitem[{Bog}02]{B02}
Andrei {Bogatyr\"ev}.
\newblock Effective approach to least deviation problems.
\newblock {\em Sb. Math.}, 193(12):1749--1769, 2002.

\bibitem[{Bog}03]{B03}
Andrei {Bogatyr\"ev}.
\newblock A combinatorial description of a moduli space of curves and of
  extremal polynomials.
\newblock {\em Sb. Math.}, 194(10):1451--1473, 2003.

\bibitem[{Bog}12]{Bbook}
Andrei {Bogatyr\"ev}.
\newblock {\em {Extremal polynomials and Riemann surfaces.}}
\newblock Springer, 2012.
\newblock Russian original, MTsNMO, Moskva, 2005.

\bibitem[Bog19]{B19}
Andrei Bogatyr\"ev.
\newblock Combinatorial analysis of the period mapping: the topology of 2d
  fibres.
\newblock {\em Sb. Math.}, 210(11):1531--1562, 2019.

\bibitem[{Bog}23]{B23}
Andrei {Bogatyr\"ev}.
\newblock Degeneration of a graph describing conformal structure.
\newblock {\em Sb. Math.}, 214(3):106--119, 2023.

\bibitem[BZ13]{ZheBu}
Vladimir {Burskii} and Alexei {Zhedanov}.
\newblock {On Dirichlet, Poncelet and Abel problems.}
\newblock {\em {Commun. Pure Appl. Anal.}}, 12(4):1587--1633, 2013.

\bibitem[CG22]{chgeCC}
Dawei Chen and Quentin Gendron.
\newblock Towards a classification of connected components of the strata of
  {{\(k\)}}-differentials.
\newblock {\em Doc. Math.}, 27:1031--1100, 2022.

\bibitem[{Che}48]{NGCheb}
Nikolaĭ {Chebotar\"ev}.
\newblock {\em Theory of algebraic functions}.
\newblock OGIZ, Moscow-Leningrad, 1948.

\bibitem[DR19]{DragoRad}
Vladimir Dragovi\'c and Milena Radnovi\'c.
\newblock Periodic ellipsoidal billiard trajectories and extremal polynomials.
\newblock {\em Communications in Mathematical Physics}, 372:183--211, 2019.

\bibitem[Gen22]{GePA}
Quentin Gendron.
\newblock Équation de {P}ell-{A}bel et applications.
\newblock {\em C. R., Math., Acad. Sci. Paris}, 360:975--992, 2022.

\bibitem[GG06]{GGBraids}
Jean-Marc Gambaudo and {\'E}tienne Ghys.
\newblock Braids and signatures.
\newblock {\em Bull. Soc. Math. Fr.}, 133(4):541--579, 2006.

\bibitem[GK10]{GrKr}
Samuel Grushevsky and Igor Krichever.
\newblock The universal {Whitham} hierarchy and the geometry of the moduli
  space of pointed {Riemann} surfaces.
\newblock In {\em Geometry of Riemann surfaces and their moduli spaces}, pages
  111--129. International Press, 2010.

\bibitem[Kol20]{KollPell}
J{\'a}nos Koll{\'a}r.
\newblock Pell surfaces.
\newblock {\em Acta Math. Hung.}, 160(2):478--518, 2020.

\bibitem[Kon91]{Kon91}
Maxim Kontsevich.
\newblock Intersection theory on the moduli space of curves.
\newblock {\em Funct. Anal. Appl.}, 25(2):123--129, 1991.

\bibitem[Kon92]{Kon92}
Maxim Kontsevich.
\newblock Intersection theory on the moduli space of curves and the matrix
  {Airy} function.
\newblock {\em Commun. Math. Phys.}, 147(1):1--23, 1992.

\bibitem[KZ96]{KhZd}
Askold Khovanskij and Smilka Zdravkovska.
\newblock Branched covers of {{\(S^2\)}} and braid groups.
\newblock {\em J. Knot Theory Ramifications}, 5(1):55--75, 1996.

\bibitem[KZ03]{KonZor}
Maxim Kontsevich and Anton Zorich.
\newblock Connected components of the moduli spaces of {A}belian differentials
  with prescribed singularities.
\newblock {\em Invent. Math.}, 153(3):631--678, 2003.

\bibitem[LO08]{FuOsHurwitz}
Fu~Liu and Brian Osserman.
\newblock The irreducibility of certain pure-cycle {Hurwitz} spaces.
\newblock {\em Am. J. Math.}, 130(6):1687--1708, 2008.

\bibitem[LZ04]{LZ}
Serguey Lando and Alexander Zvonkin.
\newblock {\em Graphs on surfaces and their applications. {Appendix} by {Don}
  {B}. {Zagier}}, volume 141 of {\em Encycl. Math. Sci.}
\newblock Berlin: Springer, 2004.

\bibitem[Mal02]{Mal}
Vladimir Malyshev.
\newblock The {Abel} equation.
\newblock {\em St. Petersbg. Math. J.}, 13(6):893--938, 2002.

\bibitem[{McM}06]{mcmtor}
Curtis {McMullen}.
\newblock {Teichmüller curves in genus two: Torsion divisors and ratios of
  sines.}
\newblock {\em {Invent. Math.}}, 165(3):651--672, 2006.

\bibitem[MP18]{MoPiHur}
Riccardo Moschetti and Gian~Pietro Pirola.
\newblock Hurwitz spaces and liftings to the {Valentiner} group.
\newblock {\em J. Pure Appl. Algebra}, 222(1):19--38, 2018.

\bibitem[Mul22]{mulSecond}
Scott Mullane.
\newblock Strata of differentials of the second kind, positivity and
  irreducibility of certain {Hurwitz} spaces.
\newblock {\em Ann. Inst. Fourier}, 72(4):1379--1416, 2022.

\bibitem[Peh93]{Peh}
Franz Peherstorfer.
\newblock Orthogonal and extremal polynomials on several intervals.
\newblock {\em J. Comput. Appl. Math.}, 48(1-2):187--205, 1993.

\bibitem[Pla14]{PlatoTorSurv}
Vladimir Platonov.
\newblock Number-theoretic properties of hyperelliptic fields and the torsion
  problem in {Jacobians} of hyperelliptic curves over the rational number
  field.
\newblock {\em Russ. Math. Surv.}, 69(1):1--34, 2014.

\bibitem[Rob64]{Rob}
Raphael Robinson.
\newblock Conjugate algebraic integers in real point sets.
\newblock {\em Math. Z.}, 84:415--427, 1964.

\bibitem[Ser19]{Serre}
Jean-Pierre Serre.
\newblock Distribution asymptotique des valeurs propres des endomorphismes de
  frobenius d’après {Abel}, {Chebyshev}, {Robinson},{{\dots}}.
\newblock In {\em S\'eminaire Bourbaki. Volume 2017/2018. Expos\'es
  1136--1150}, pages 379--426. Soci{\'e}t{\'e} Math{\'e}matique de France,
  2019.

\bibitem[{Str}84]{strebel}
Kurt {Strebel}.
\newblock {\em {Quadratic differentials}}, volume~5 of {\em {Ergebnisse der
  Mathematik und ihrer Grenzgebiete. 3. Folge}}.
\newblock Springer, 1984.

\bibitem[SY92]{SYu}
Mikhail {Sodin} and Peter {Yuditskij}.
\newblock {Functions least deviating from zero on closed subsets of the real
  axis.}
\newblock {\em {Algebra Anal.}}, 4(2):1--61, 1992.
\newblock English translation, St. Petersburg Math. J., 4.2, p. 201--249, 1993.

\bibitem[Thu79]{ThGT3}
William Thurston.
\newblock The geometry and topology of three-manifolds, 1979.
\newblock Freely available at: http://library.msri.org/books/gt3m/PDF/13.pdf.

\bibitem[Vas95]{Vas}
Viktor Vassiliev.
\newblock {\em Ramified integrals, singularities and lacunas}, volume 315 of
  {\em Math. Appl., Dordr.}
\newblock Kluwer Academic Publishers, 1995.

\bibitem[Waj96]{WajHur}
Bronislaw Wajnryb.
\newblock Orbits of {Hurwitz} action for coverings of a sphere with two special
  fibers.
\newblock {\em Indag. Math., New Ser.}, 7(4):549--558, 1996.

\bibitem[Zan14]{ZannierUnli}
Umberto Zannier.
\newblock Unlikely intersections and {Pell}'s equations in polynomials.
\newblock In {\em Trends in contemporary mathematics}, pages 151--169. Cham:
  Springer, 2014.

\bibitem[Zol77]{Zol}
Egor Zolotarev.
\newblock On the application of elliptic functions to questions of maxima and
  minima.
\newblock M{\'e}l. math de {St}. {P{\'e}tersb}. 15, 1877.

\end{thebibliography}

\bigskip
\noindent
\small{Andrei Bogatyrev\\
Institute for Numerical Mathematics,\\
Russian Academy of Sciences,\\
119991 Russia, Moscow, ul. Gubkina 8\\
{\em email:} ab.bogatyrev@gmail.com} 
\smallskip
\par
\noindent
\small{Quentin Gendron\\
Instituto de Matem\'{a}ticas de la UNAM\\
Ciudad Universitaria, CDMX, 04510, M\'{e}xico\\
{\em email:} quentin.gendron@im.unam.mx} 
 \end{document}